\documentclass[a4paper]{amsart}

\usepackage{amsmath}
\usepackage{amsthm}
\usepackage{enumerate}
\usepackage[varg]{txfonts} 
\usepackage[arrow, curve, matrix, all]{xy}
\usepackage[]{graphicx}
\usepackage{siunitx}
\usepackage[]{color}
\usepackage{mathrsfs}
\usepackage{float}

\theoremstyle{definition}
\newtheorem{dfn}{Definition}[section]

\newtheorem{rem}[dfn]{Remark}

\newtheorem{prop}[dfn]{Proposition}
\newtheorem{lem}[dfn]{Lemma}
\newtheorem{thm}[dfn]{Theorem}

\newcommand{\M}{\mathcal{M}}
\renewcommand{\H}{\mathcal{H}}

\newcommand{\st}{\mathrm{st}\,}

\numberwithin{equation}{section}

\definecolor{green}{rgb}{0.0000,0.4000,0.0000}
\definecolor{purple}{rgb}{0.5000,0.0000,0.5000}

\title{The Morse complex of the wedge of two extended star graphs and a path}
\author{Shuma Komatsu}
\address{Faculty of Mathematics, Shinshu University, 3-1-1 Asahi, Matsumoto, Nagano 390-8621, Japan}
\email{24ss107e@shinshu-u.ac.jp}
\begin{document}
\maketitle

\begin{abstract}
  Let $K$ be a finite abstract simplicial complex.
  The Morse complex of $K$, first introduced by Chari and Joswig, is defined as the simplicial complex constructed from all gradient vector fields on $K$.
  In this paper, we determine the homotopy type of the Morse complex of the graph obtained by attaching the center vertices of two extended star graphs to  different endpoints of a path. 
  These results generalize recent work of Donovan and Scoville \cite{DS23}.
\end{abstract}

\section{introduction}

Forman's discrete Morse theory \cite{For98} provides a method for constructing a gradient vector field $V$ on $K$.
He showed that $K$ is homotopy equivalent to a CW complex with exactly one $p$-dimension cell for each critical simplex of index $p$ on $V$.
Subsequently, in 2005, The Morse complex of $K$, denoted $\M(K)$ was introduced by Chari and Joswig \cite{CJ05} as the simplicial complex of all gradient vector fields on $K$. 
In general, there are only a few complexes for which the homotopy types of the Morse complexes have been computed, and it remains difficult to compute the homotopy type of $\M(K)$ even for very small complexes $K$ such as the 3-simplex.
In spite of this, Kozlov \cite{Koz99} successfully determined the homotopy type of the Morse complexes of path and cycle graph.
Furthermore, Donovan, Lin and Scoville \cite{DLS22} determined the homotopy type of the Morse complex of a cycle wedged with a leaf by using strong collapses and Hasse diagram.
In \cite{Bar13}, Barmak introduced the \emph{star cluster} of a simplex in $K$ and proved that if $K$ is a flag complex, then the star cluster is collapsible.
This provided a tool for studying the topology of Morse complexes.
Another important tool for determining the homotopy type of the Morse complexes is the \emph{Cluster Lemma}.
This result was independently obtained by both Jonsson \cite{Jon08} and Hersh \cite{Her05}, and provides a very simple and convenient way to put a gradient vector field on a simplicial complex by combining together gradient vector fields on a disjoint decomposition of the complex.
For example, Donovan and Scoville \cite{DS23} showed that the Morse complex of a tree is a flag complex, and using star clusters and the Cluster Lemma, they determined the homotopy type of the Morse complex of the \emph{extended star graph} denoted $S_{0, n}$.

This paper provides a computational example that extends the result of \cite{DS23}.
In Section 3, we compute the homotopy type of the Morse complex of the extended star graph $S_{1, n}$, using strong collapses and the Hasse diagram (Theorem \ref{thm:3b}).
In Section 4, we compute the homotopy type of the Morse complex of the graph $P_t\vee S_{0, n}\vee S_{0, l}$, which is obtained by taking the wedge of a path $P_t$ of length $t$ with the center vertices of two extended star graphs at its endpoints (Theorem \ref{thm:main theorem}).
The computation is carried out using star clusters and the Cluster Lemma. 
This is the main result of this paper.
Finally, as in Section 3, we compute the homotopy type of the Morse complex of $P_t\vee S_{k, n}\vee S_{k', l}$ for $k, k'\in\{0, 1\}$ using strong collapses and Hasse diagram (Theorem \ref{thm:4b}, \ref{thm:4c}).

\section*{Acknowledgments}
  The author would like thank Professor Keiichi Sakai for his thoughtful guidance and constructive suggestions.

\section{Preliminaries}

An \emph{abstract simplicial complex} is a pair $(K, \Delta)$, where $K$ is a set and $\Delta\subset2^K$, satisfying the following property:
\begin{enumerate}
  \item if $\sigma\in\Delta, \tau\subset\sigma$, then $\tau\in\Delta$.
  \item $\varnothing\in\Delta$.
\end{enumerate}
The pair $(K, \Delta)$ is often simply denoted by $K$.
An element $\sigma\in\Delta$ is called a \emph{simplex}.
A $p$-simplex is a simplex $\sigma$ such that $|\sigma|=p+1$.

In the following, let $K$ be a finite abstract simplicial complex.
Write $\sigma^{(p)}$ if $\sigma$ has dimension $p$ and $\tau<\sigma$ if $\tau$ is a face of $\sigma$.
In order to describe the Morse complexes, we will need the following.

\begin{dfn}
  We write $V(K)\coloneqq\{\sigma^{(0)}\in K\}$ and $E(K)\coloneqq\{\sigma^{(1)}\in K\}$, and refer to them as the \emph{vertex set} and \emph{edge set} of $K$, respectively.
  A 0-simplex adjacent to $v\in V(K)$ is called a \emph{neighbor} of $v$.
  The \emph{degree} of a vertex $v\in V(K)$, denoted $\deg(v)$ is the number of neighbors of $v$.
  If $\deg(v)=1$ and the neighbor of $v$ is $w$, we say that $\{v, vw\}$ is a \emph{leaf}.
\end{dfn}

\begin{dfn}
  A \emph{matching} in $K$ is a subset $M\subset E(K)$ of the edge set such that no two edges in $M$ share a common vertex.
\end{dfn}

\begin{dfn}
  A simplex of $K$ that is not properly contained in any other simplex of $K$ is called a \emph{facet} of $K$.
\end{dfn}

\begin{dfn}
  Let $u\ge0$ be an integer and define $[v_u]\coloneqq\{v_0, v_1,\dots, v_u\}$.
  Let $P_u$ denote the simplicial complex on $[v_u]$ with facets $\{v_0, v_1\}, \{v_1, v_2\},\dots ,\{v_{u-1}, v_u\}$.
  i.e., the \emph{path} of length $u$.
\end{dfn}

\begin{dfn}[cf. \cite{DLS22,DS23}]\label{dfn:discrete vector field}
  Let $K$ be a simplicial complex. 
  A collection $V$ of pairs $(\sigma^{(p)}, \tau^{(p+1)})$ of simplices of codimension 1 is called a \emph{discrete vector field} if each simplex of $K$ appears in at most one such pair.
  Any pair in $(\sigma, \tau)\in V$ is called a \emph{regular pair}, and $\sigma, \tau$ are called \emph{regular}.
  If $(\sigma^{(p)}, \tau^{(p+1)})\in V$, we say that $p+1$ is the \emph{index} of the regular pair.
  Any simplex in $K$ which does not appear in any pair of $V$ called \emph{critical}.
\end{dfn}

\begin{dfn}[cf. \cite{DLS22,DS23}]
  Let $V$ be a discrete Morse vector field  on a simplicial complex $K$.
  A \emph{$V$-path} or \emph{gradient path} is a sequence of simplices 
  \begin{equation}
    \alpha_0^{(p)}, \beta_0^{(p+1)}, \alpha_1^{(p)}, \beta_1^{(p+1)}, \alpha_2^{(p)},\dots, \beta_{k-1}^{(p+1)}, \alpha_k^{(p)}
  \end{equation}
  of $K$ such that $(\alpha_i^{(p)}, \beta_i^{(p+1)})\in V$ and $\beta_i^{(p+1)}>\alpha_{i+1}^{(p)}\ne\alpha_i^{(p)}$ for $0\leq i \leq k-1$.
  If $k\ne0$, then the $V$-path is called \emph{non-trivial}.
  A $V$-path is said to be \emph{closed} if $\alpha_k^{(p)}=\alpha_0^{(p)}$.
  A discrete vector field $V$ which contains no non-trivial closed $V$-paths is called a \emph{gradient vector field}.
  We sometimes use $f$ to denote a gradient vector field.
  \end{dfn}

If the gradient vector field $f$ consists of only a single element, we call $f$ a \emph{primitive} gradient vector field.
We often denote a primitive gradient vector field $\{(u, uv)\}$ with $p=0$ by $(u)v$.

If $f, g$ are two gradient vector fields on $K$, write $g\leq f$ whenever the regular pairs of $g$ are also regular pairs of $f$.

\begin{dfn}[cf. \cite{DLS22,DS23}]\label{dfn:Morse cplx}
  The \emph{Morse complex} of $K$, denoted $\M(K)$, is the simplicial complex whose vertices are in one-to-one correspondence to the primitive vector fields of $K$, and an $(n+1)$-tuple of primitive gradient vector fields forms an $n$-simplex if any two of them share no common simplex of $K$.
  A gradient vector field $f$ is then associated with all primitive gradient vector fields $f\coloneqq\{f_0,\dots, f_n\}$ with $f_i\leq f$ for all $0\leq i\leq n$.
\end{dfn}

Let $\H(K)$ be the Hasse diagram of $K$, defined as a directed graph whose vertices correspond to the simplices of $K$, and which has a directed edge from $\sigma$ to $\tau$ whenever $\tau<\sigma$ and there is no simplex $\nu$ such that $\tau<\nu<\sigma$.
From definition \ref{dfn:discrete vector field}, we see that the non-critical pairs form a matching in the Hasse diagram.
If we reverse the orientation of the arrows in this matching, it can be shown that the resulting directed graph obtained is acyclic.
We will call such a matching in the Hasse diagram of the given simplicial complex \emph{acyclic matching}.

We write $\sigma\longleftrightarrow\tau$ to denote that the simplices  $\sigma$ and $\tau$ are matched.

\begin{rem}\label{rem:Morse cplx}
  By using Hasse diagrams, we can give an eqivalent definition of the Morse complex as in Definition \ref{dfn:Morse cplx}:
  The Morse complex $\M(K)$ of $K$ can be defined as a simplicial complex whose vertices are in one-to-one correspondence to the edges of $\H(K)$, and $n$-simplices correspond to  acyclic matching in $\H(K)$ consisting of $n+1$ edges.
\end{rem}

We use the notation $K-\{v^\prime\}\coloneqq\{\sigma\in K\mid v^\prime\notin\sigma\}$.

\begin{dfn}[cf. \cite{DLS22,DS23}]
  If $v$ dominates $v^\prime$, then the removal of $v^\prime$ from $K$ is called an \emph{elementary strong collapse} and is denoted by $K\searrow\searrow K-\{v^\prime\}$.
\end{dfn}

In particular, $K\simeq K-\{v'\}$ holds.

\begin{dfn}[cf. \cite{DLS22,DS23}]
  Let $K$ and $L$ be simplicial complexes.
  If there is a sequence of strong collapses from $K$ to $L$, then $K$ and $L$ are said to have the same \emph{strong homotopy type}.
  In the case when $L=*$, then $K$ is said to have the \emph{strong homotopy type of a point}.
  If there is a sequence of only strong collapses from $K$ to a point, $K$ is \emph{strongly collapsible}.
\end{dfn}

One construction that is particularly well behaved with respect to strong collapses is the join.

\begin{prop}[{\cite[Proposition 4.2]{DLS22}}]
  Let $K, L$ be connected simplicial complexes each with at least one edge.
  Then,
  \begin{equation}
    \M(K\sqcup L)=\M(K)*\M(L)
  \end{equation}
  where $\M(K)*\M(L)$ denotes the join of $\M(K)$ and $\M(L)$.
\end{prop}

The following lemma will be useful in the subsequent arguments when the simplicial complex has a leaf.

\begin{lem}[{\cite[Lemma 3.9]{DS23}}]\label{lem:dominated vertex}
  Let $K$ be a simplicial complex with leaf $\{a, ab\}$ and $c$ a neighbor of $b$ not equal to $a$.
  Then, $(b)c$ is dominated in $\M(K)$ by $(a)b$.
\end{lem}

\begin{prop}[{\cite[Proposition 3.5]{DLS22}}]\label{prop:strong collapsibility}
  If $K$ has two leaves sharing a common vertex, then $\M(K)$ is strongly collapsible.
\end{prop}

Let $K\vee_v P_u$ denote attaching one endpoint of the path $P_u$ to a vertex $v\in K$.

\begin{lem}[{\cite[Lemma 3.18]{DS23}}]\label{lem:suspension}
  For any $t\ge1$,
  \begin{equation}
    \M(K\vee_vP_{3t})\simeq\Sigma^{2t}\M(K)
  \end{equation}
  where $\Sigma^{2t}\M(K)$ denotes the $2t$-fold suspension of $\M(K)$.
\end{lem}

\section{extended star graphs}

\begin{dfn}[cf. \cite{Bar13, DS23}]
  A simplicial complex $K$ is a \emph{flag complex} if for each non-empty set of vertices $\sigma$ such that $\{v_i, v_j\}\in K$ for every $v_i, v_j\in\sigma$, we have that $\sigma\in K$.
\end{dfn}

\begin{dfn}[cf. \cite{Bar11}]
  If $K$ is a simplicial complex and $v$ is a vertex of $K$, the \emph{star} of $v$ in $K$ is the subcomplex $\st_K(v)\subseteq K$ of simplices $\sigma\in K$ such that $\sigma\cup\{v\}\in K$.
\end{dfn}

\begin{figure}[htbp]
  \centering
  %star.tex
  \includegraphics[scale=0.13]{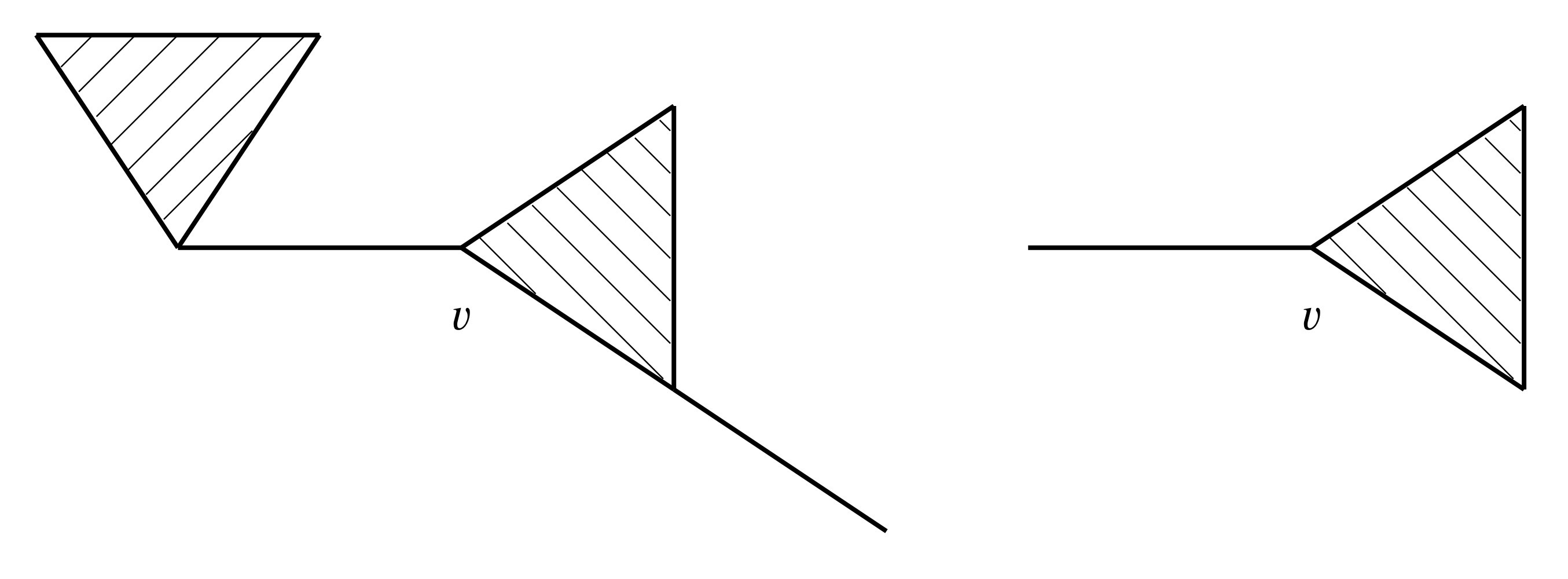}
  \caption{A simplicial complex $K$ and a vertex $v\in K$ at the left and the star $\st_K(v)$ of $v$ at the right.}
  \label{fig:star}
\end{figure}

\begin{dfn}[cf. \cite{Bar13, DS23}]
  Let $\sigma$ be a simplex of a simplicial complex $K$.
  We define the \emph{star cluster} of $\sigma$ in $K$ as the subcomplex
  \begin{equation}
    \mathrm{SC}_K(\sigma)=\bigcup_{v\in\sigma}\st_K(v).
  \end{equation}
\end{dfn}

\begin{figure}[htbp]
  \centering
  %star_cluster.tex
  \includegraphics[scale=0.13]{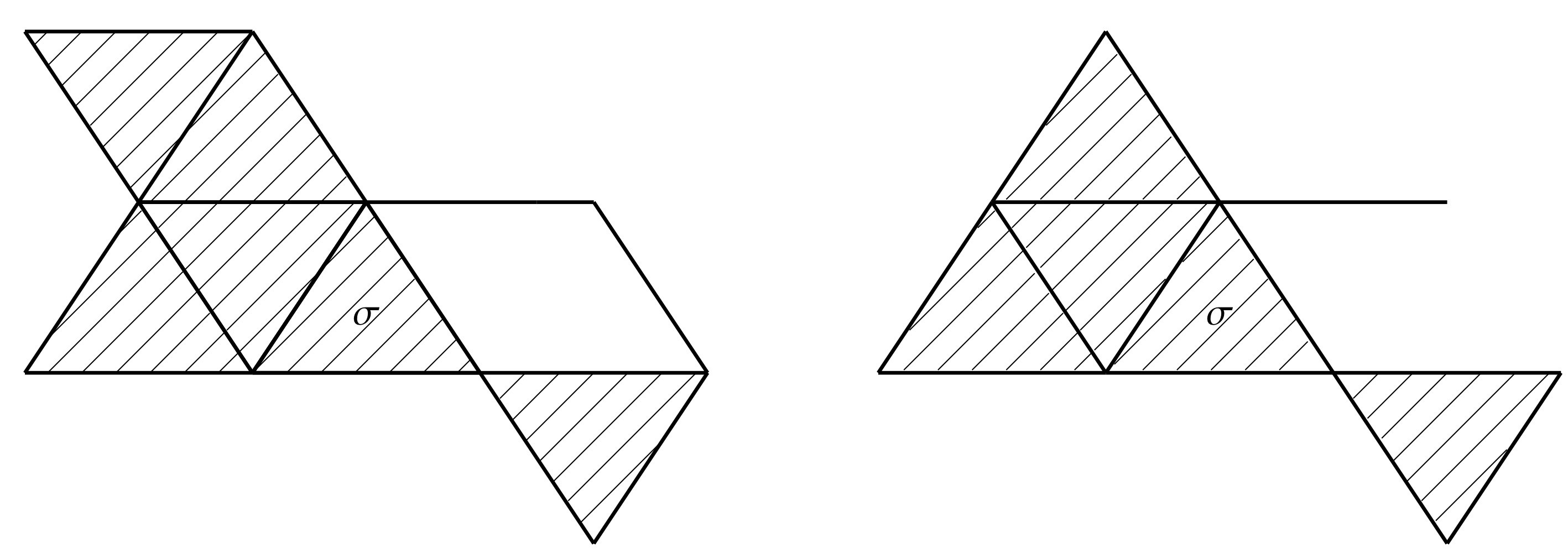}
  \caption{A 2-dimension flag complex at the left and the star cluster of a simplex $\sigma$ at the right.}
  \label{fig:star cluster}
\end{figure}

\begin{prop}[{\cite[Lemma 3.2]{Bar13}}]\label{prop:collapsible}
  The star cluster of a simplex in a flag complex is collapsible.
\end{prop}

\begin{lem}[{\cite[Lemma 3.5]{DS23}}]\label{lem:flag complex}
  $\M(K)$ is a flag complex if and only if $K$ is a tree.
\end{lem}

Proposition \ref{prop:collapsible} and Lemma \ref{lem:flag complex} are one of the main tools we use in this paper.
The other tool is the following Lemma:

\begin{lem}[{\cite[Lem 4.2]{Jon08} and \cite[Lemma 4.1]{Her05}}][Cluster Lemma]\label{lem: cluster lemma}
  Let $\Delta$ be a simplicial complex which is decomposed into a union of collections $\Delta_\sigma$ of simplices, indexed by the elements $\sigma$ in partially ordered set $P$ which has a unique minimal element.
  Furthermore, assume that this decomposition is as follows:
  \begin{enumerate}
    \item Each simplex belongs to exactly one $\Delta_\sigma$.
    \item For each $\sigma\in P$, $\bigcup_{\tau\leq\sigma}\Delta_\tau$ is a subcomplex of $\Delta$.
  \end{enumerate}
  For each $\sigma\in P$, let $M_\sigma$ be an acyclic matching in $\Delta_\sigma$.
  Then, $\bigcup_{\sigma\in P}M_\sigma$ is an acyclic matching on $\Delta$.
\end{lem}

Lemma \ref{lem: cluster lemma} shows how to construct an acyclic matching on an entire complex by combining acyclic matchings on its individual parts. The crucial information is what remains unmatched—that is, the critical simplices. In some cases, a specific set of critical simplices uniquely determines the homotopy type of the original complex. This is demonstrated in Forman's classical result.

\begin{thm}[{\cite[Corollary 3.5]{For98}}]\label{thm:Morse theorem}
  Let $K$ be a simplicial complex and $M$ an acyclic matching on $K$ with $m_i$ critical simplices of dimension $i$.
  Then $K$ has the homotopy type of a CW complex with exactly $m_i$ cells of dimension $i$.
  In particular, if $m_0=1, m_n=k$, and $m_j=0$ for all $j\ne0, n$, then $K$ has the homotopy type of a $k$-fold wedge of $S^n$.
\end{thm}

\begin{dfn}[cf. \cite{DS23}]\label{dfn:extended star graph}
  An \emph{extended star graph}, denoted by $S_{m, n}$, is the graph that is a one-point union of $m$ paths of length 1 and $n$ paths of length 2. The vertex $c$ shared by these paths is called the \emph{center}.
\end{dfn}

\begin{dfn}[cf. \cite{RS20}]
  Let $v$ be a vertex in $K$, and $V$ be a discrete vector field on $K$.
  We say that \emph{$V$ roots in $v$} or \emph{$V$ is rooted in $v$} if $v$ is the unique critical simplex of $V$.
  Such a vertex is called the \emph{root} of $V$.
\end{dfn}

\begin{thm}[{\cite[Theorem 3.12]{DS23}}]\label{thm:3a}
  For any $n\ge1$,
  \begin{equation}
    \M(S_{0, n})\simeq(S^n)^{\vee(n-1)}.
  \end{equation}
\end{thm}

\begin{proof}
  Let $c$ be the center of $S_{0, n}$ and let $\{a_ib_i, b_i\}$ be the leaf of each extended leaf of length 2, $i=1, 2,\dots, n$.
  Let $\sigma_0$ be the star cluster in $\M(S_{0, n})$ of the gradient vector field rooted in $c$ (see Figure \ref{fig:star cluster of extended star graph}).
  Now define $\Delta_0\coloneqq\sigma_0$, and $\Delta_1\coloneqq\M(S_{0, n})-\sigma_0$.
  Clearly, $\Delta_0\cup\Delta_1=\M(S_{0, n})$ so we can apply Lemma \ref{lem: cluster lemma}.
  We define an acyclic matching on $\Delta_0, \Delta_1$ as follows:

  First, $\Delta_0$ is collapsible by Proposition \ref{prop:collapsible} and Lemma \ref{lem:flag complex} so there is an acyclic matching on $\Delta_0$ with a single critical 0-simplex.

  Next, we consider $\Delta_1$.
  Any element of $\Delta_1$ contains the form $\bigcup_{i=1}^n\{(a_i)b_i\}$, and possibly one of $(c)a_i$.
  Therefore, $\Delta_1$ has $n+1$ elements.
  Match $\bigcup_{i=1}^n\{(a_i)b_i\}$ with $\bigcup_{i=1}^n\{(a_i)b_i\}\cup\{(c)a_1\}$ (see Figure \ref{fig:3a}).
  Then, there are $n-1$ unmatched $n$-simplices which contains $(c)a_i$ for $i=2, 3,\dots, n$.
  Thus, by Theorem \ref{thm:Morse theorem}, $\M(S_{0, n})\simeq (S^n)^{\vee(n-1)}$.
\end{proof}

\begin{figure}[htbp]
  \centering
  %star_cluster_of_extended_star_graph.tex
  \includegraphics[scale=0.13]{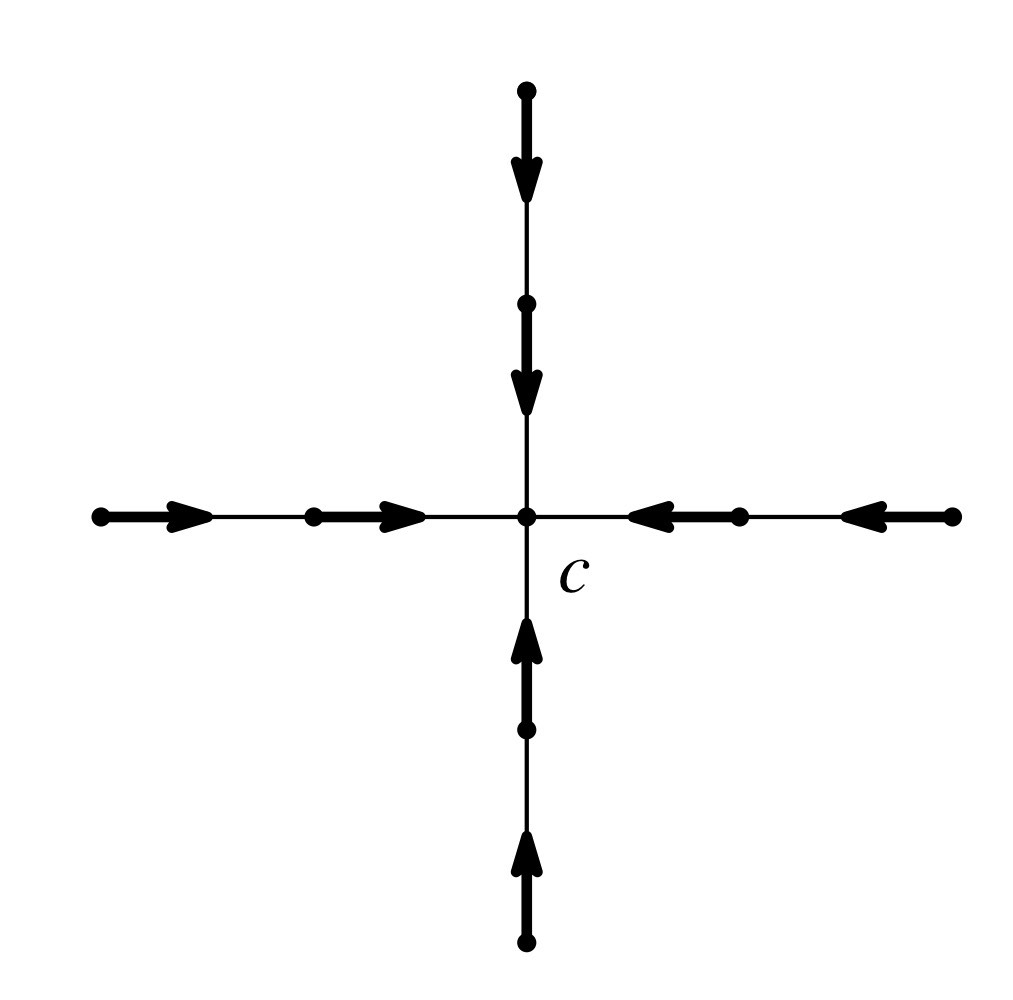}
  \caption{$\Delta_0$ is the union of the star of each arrow.}
  \label{fig:star cluster of extended star graph}
\end{figure}

\begin{figure}[htbp]
  \centering
  \includegraphics[scale=0.13]{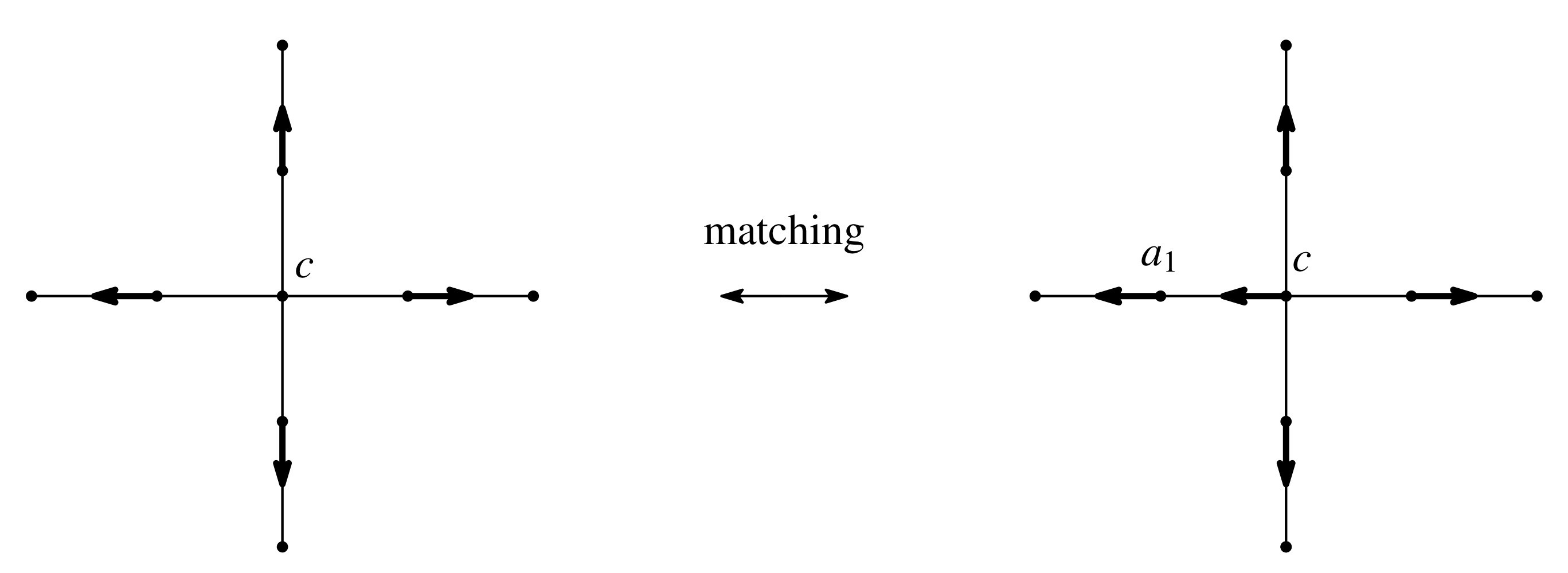}
  \caption{Match $\bigcup_{i=1}^n\{(a_i)b_i\}$ with $\bigcup_{i=1}^n\{(a_i)b_i\}\cup\{(c)a_1\}$.}
  \label{fig:3a}
\end{figure}

\begin{dfn}[cf. \cite{DLS22,DS23}]\label{dfn:poset}
  Given a finite poset $P$, construct a simplicial complex $f(P)$ as follows:
  The vertex set of $f(P)$ is the edge set of the Hasse diagram $\H(P)$ of $P$.
  Then, let $\sigma=e_1e_2\dots e_k$ be a simplex of $f(P)$ if and only if the edges $e_1, e_2,\dots, e_k$ form an acyclic matching of $P$.
\end{dfn}

Viewing the Morse complex as defined in Remark \ref{rem:Morse cplx}, note that for any simplicial complex $K$, $\M(K)\simeq f(\H(K))$.
The following theorem was not proved in \cite{DS23}.

\begin{thm}\label{thm:3b}
  For any $n\ge1$,
  \begin{equation}
    \M(S_{1, n})\simeq S^n.
  \end{equation}
\end{thm}

\begin{figure}[htbp]
  \centering
  %S_1_n.tex
  \includegraphics[scale=0.10]{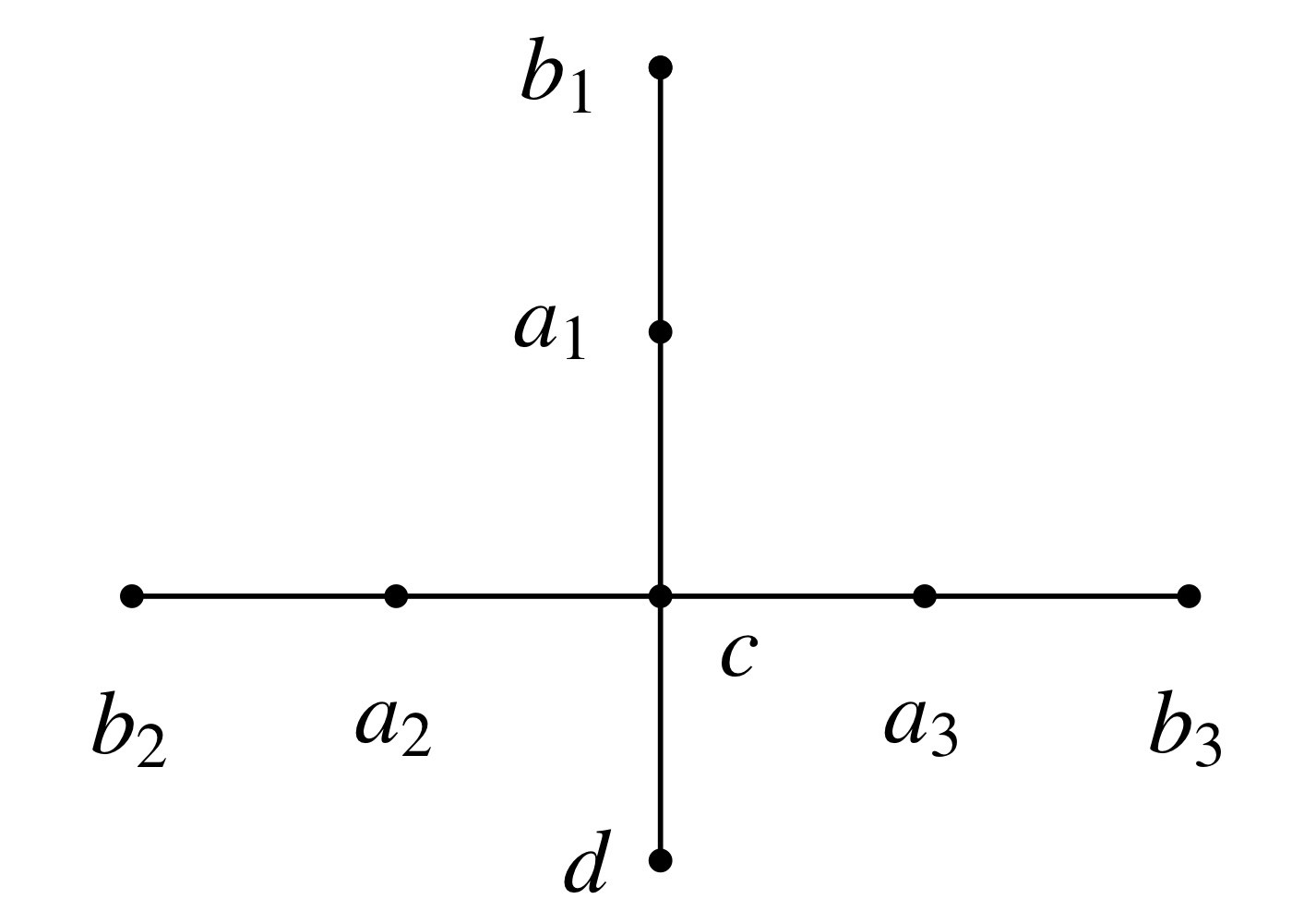}
  \caption{$S_{1, 3}$. By Theorem \ref{thm:3b}, we see that $\M(S_{1, 3})\simeq S^3$.}
  \label{fig:3b}
\end{figure}

\begin{proof}
  Let $c$ be the center of $S_{1, n}$ and let $\{cd, d\}$ be the leaf of length 1 and $\{a_ib_i, b_i\}$ be the leaves of length 2, $i=1, 2,\dots, n$.
  By Lemma \ref{lem:dominated vertex}, $(d)c$ dominates $(c)a_i$.
  In the corresponding Hasse diagram $\H(S_{1, n})$, this corresponds to the removal of the edge between $c$ and $ca_i$.
  Since these are all the edges connected to $c$ in $S_{0, n}$, the Hasse diagram consists of the Hasse diagram of the leaf $ca$, together with $n$ copies of the Hasse diagram of $P_2$, obtained by connecting $c$, $a_i$ and $b_i$ in this order, with $c$ removed.
  The entire Hasse diagram is $(\H(P_2)-c)^{\sqcup n}\sqcup\H(P_1)$ (see Figure \ref{fig:Hasse diagram1}).
  Therefore,
  \begin{equation}
    \begin{split}
      \M(S_{1, n})&\searrow\searrow f((\H(P_2)-c)^{\sqcup n}\sqcup\H(P_1))\\
      &\simeq f(\H(P_2)-c)^{*n}*f(\H(P_1))\\
      &\simeq (S^0)^{*n}*S^0\\
      &\simeq S^n.
    \end{split}
  \end{equation}
\end{proof}

\begin{figure}[htbp]
  \centering
  %Hasse_diagram_S_1_n.tex
  \includegraphics[scale=0.10]{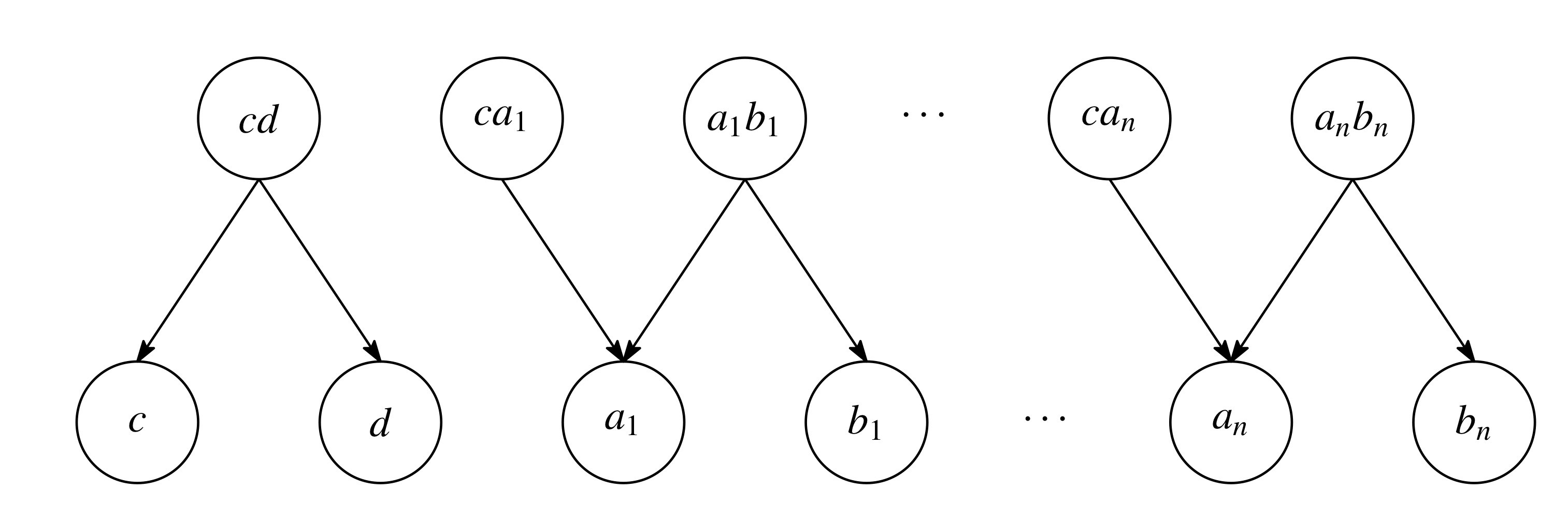}
  \caption{The Hasse diagram $\H(S_{1, n})$ after removing the edge between $c$ and $ca_i$ ($i=1, 2,\dots, n$).}
  \label{fig:Hasse diagram}
\end{figure}

By Proposition \ref{prop:strong collapsibility}, $M(S_{n, m})$ is strongly collapsible if $n\geq2$.
Moreover, if the extended star graph has an extended leaf of length at least 3, the Morse complex can be determined by applying Lemma \ref{lem:suspension}, Theorem \ref{thm:3a} and Theorem \ref{thm:3b}.

\section{main theorem}

Let $P_t$ be the path on $t+1$ vertices, and let $S_{m, n}$ and $S_{k, l}$ be the extended star graphs.
We denoted by $P_t\vee_{v_0} S_{m, n}\vee_{v_t} S_{k, l}$ the graph obtained by attaching the center vertex of each extended star graph to a different endpoint of $P_t$, namely $v_0$ and $v_t$, respectively.
By Proposition \ref{prop:strong collapsibility}, $\M(P_t\vee S_{m, n}\vee S_{k, l})$ is strongly collapsible if $m\geq2$ or $k\geq2$.

\begin{figure}[htbp]
  \centering
  %example.tex
  \includegraphics[scale=0.12]{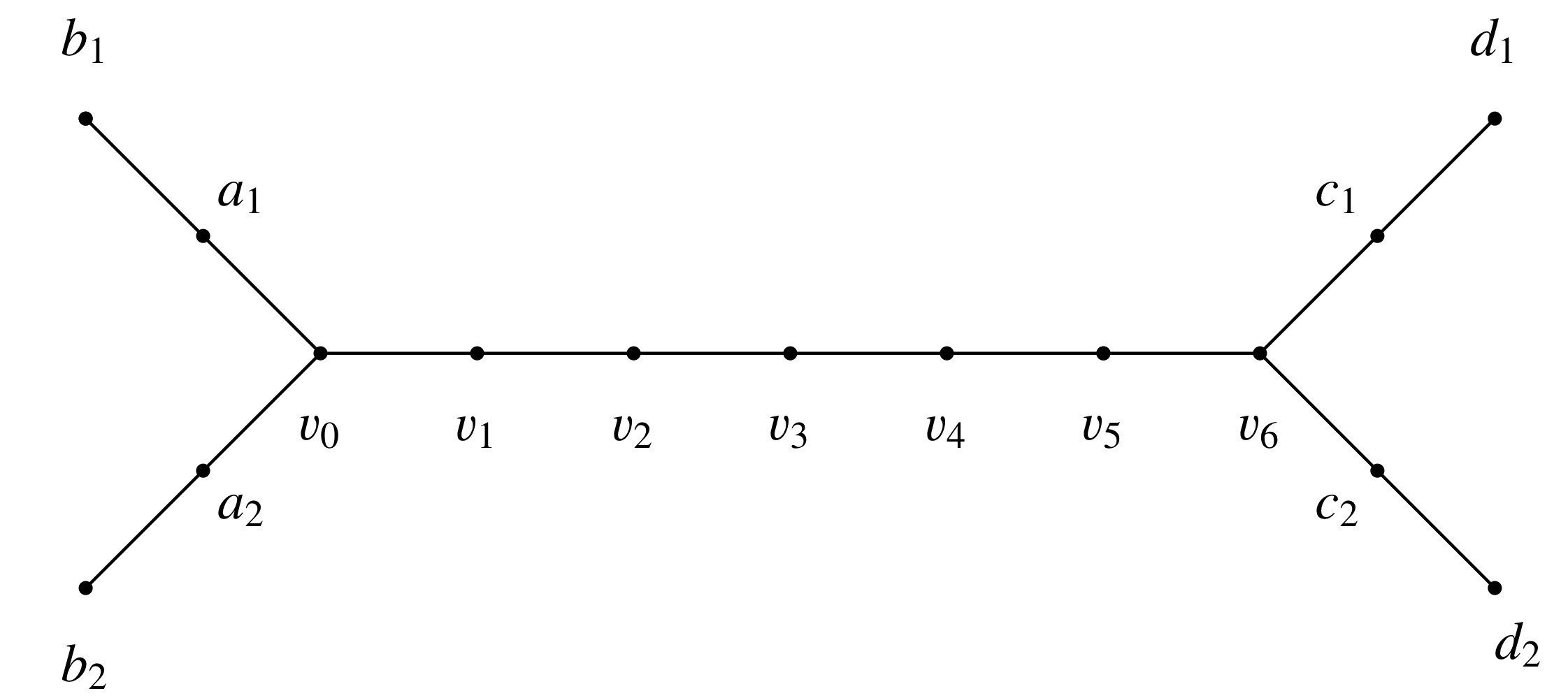}
  \caption{$P_6\vee S_{0, 2}\vee S_{0,2}$. By Theorem \ref{thm:main theorem}, we see that $\M(P_6\vee S_{0, 2}\vee S_{0, 2})\simeq S^8\vee S^8\vee S^8.$}
  \label{fig:my_label}
\end{figure}

\begin{thm}\label{thm:main theorem}
  For any $t, n, l$, we have
  \begin{equation}
    \M(P_t\vee S_{0, n}\vee S_{0, l})\simeq
    \begin{cases}
      (S^{n+l+2u})^{\vee(n+l-1)} &(t=3u)\\
      (S^{n+l+2u+1})^{\vee(nl-1)} &(t=3u+1)\\
      (S^{n+l+2u+2})^{\vee\{(n-1)(l-1)-1\}} &(t=3u+2)
    \end{cases}
  \end{equation}
\end{thm}

\begin{proof}
  Let $\{a_ib_i, b_i\}$ and $\{c_sd_s, d_s\}$ be the leaves of extended leaf of length 2 in $S_{0, n}$ and $S_{0, l}$, respectively ($i=1, 2,\dots, n, s=1, 2,\dots, l$).
  We label the vertices of $P_t$ as $V(P_t)=\{\dots, v_{3u-2}v_{3u-1}, v_{3u-1}, v_{3u}\}$.
  Then, $E(P_t)=\{\dots, v_{3u-1}v_{3u}\}$.

  Let $\sigma_0$ be the star cluster of the gradient vector field rooted in $v_0$ if $t=3u$, in $v_{-1}$ if $t=3u+1$ and in $v_{-2}$ if $t=3u+2$ (see Figure \ref{fig:rooted}).

  \begin{figure}[htbp]
    \centering
    %rooted.tex
    \includegraphics[scale=0.11]{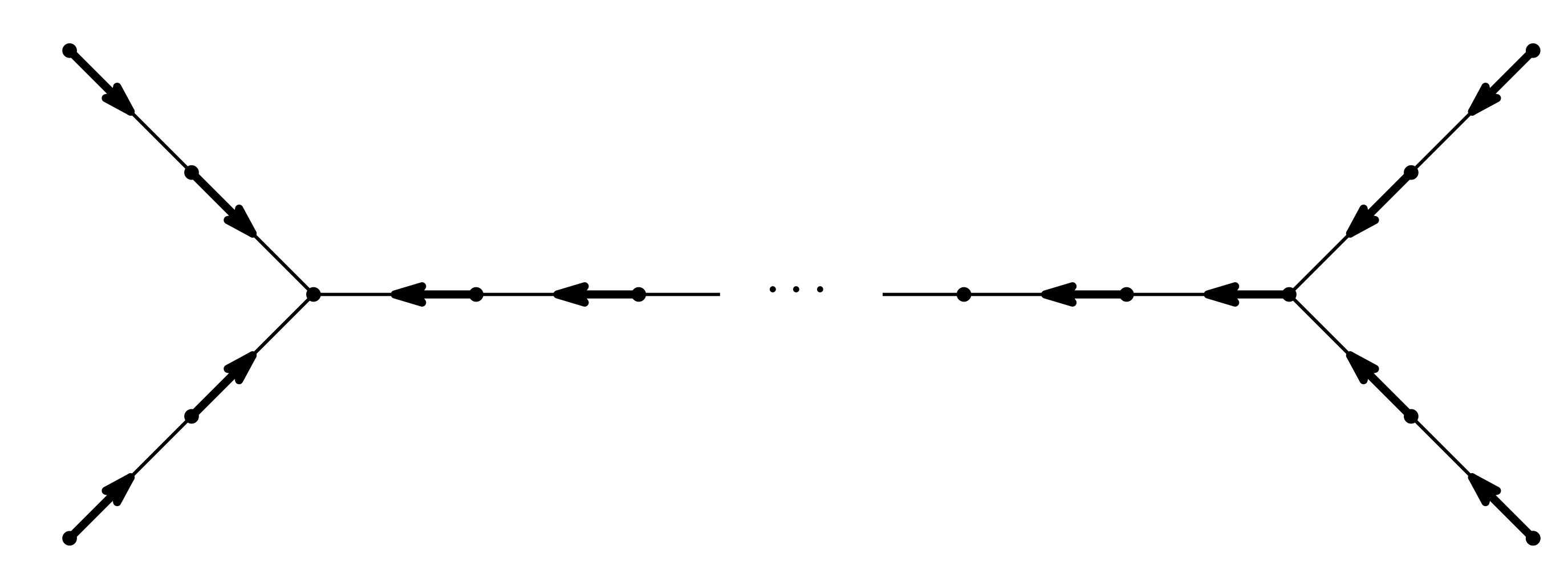}
    \caption{$\sigma_0$ is the union of the star of each arrow.}
    \label{fig:rooted}
  \end{figure}

  Now define $\Delta_0\coloneqq\sigma_0$ and $\Delta_1\coloneqq\M(P_t\vee S_{0, n}\vee S_{0, l})-\sigma_0$.
  Clearly, $\Delta_0\cup \Delta_1=\M(P_t\vee S_{0, n}\vee S_{0, l})$, so we can apply Lemma \ref{lem: cluster lemma}.
  We define an acyclic matching on $\Delta_0, \Delta_1$ as follows:

  First, $\Delta_0$ is collapsible by Proposition \ref{prop:collapsible} and Lemma \ref{lem:flag complex} so there is an acyclic matching on $\Delta_0$ with a single critical 0-simplex.

  Now consider $\Delta_{1}$.
  Any simplex that does not contain $(a_i)b_i$ is contained in $\st((b_i)a_i)$, and any simplex that does not contain $(c_s)d_s$ is contained in $\st((d_s)c_s)$; that is, they are contained in $\Delta_0$, so every simplex of $\Delta_1$ contains the form 
  \begin{equation}
  V=\left(\bigcup_{i=1}^n\{(a_i)b_i\}\right)\cup\left(\bigcup_{s=1}^l\{(c_s)d_s\}\right).
  \end{equation}
  We match the following two simplices as follow:
  \begin{equation}\label{eq:path matching}
    W\cup\{(v_{3k})v_{3k+1}, (v_{3k+1})v_{3k+2}, (v_{3k+2})v_{3k+3}\} \longleftrightarrow W\cup\{(v_{3k})v_{3k+1}, (v_{3k+2})v_{3k+3}\}\quad (k=0,1,\dots, u-1)
  \end{equation}
  where $W$ is any simplex of $\M(P_t\vee S_{0, n}\vee S_{0, l})$ such that both of these two simplices are contained in $\Delta_1$ (see Figure \ref{fig:matching}).

  \begin{figure}[htbp]
    \centering
    %matching.tex
    \includegraphics[scale=0.08]{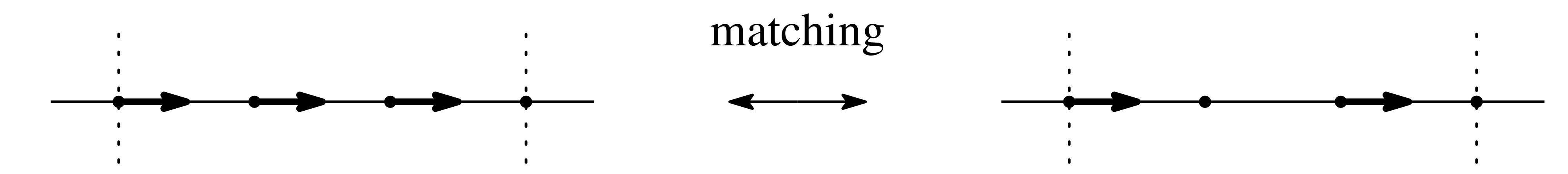}
    \caption{Match $W\cup\{(v_{3k})v_{3k+1}, (v_{3k+1})v_{3k+2}, (v_{3k+2})v_{3k+3}\}$ with $W\cup\{(v_{3k})v_{3k+1}, (v_{3k+2})v_{3k+3}\}$.}
    \label{fig:matching}
  \end{figure}
  Let us define the following collection of simplices:
  \begin{eqnarray}
    &R=\displaystyle\bigcup_{k=0}^{u-1}\{(v_{3k+1})v_{3k+2}, (v_{3k+2})v_{3k+3}\}\\
    &L=\displaystyle\bigcup_{k=0}^{u-1}\{(v_{3k})v_{3k+1}, (v_{3k+1})v_{3k+2}\}\\
    &B_j=\left(\displaystyle\bigcup_{k=0}^{j}\{(v_{3k+1})v_{3k+2}, (v_{3k+2})v_{3k+3}\}\right)\cup\left(\displaystyle\bigcup_{k=j+1}^{u-1}\{(v_{3k})v_{3k+1}, (v_{3k+1})v_{3k+2}\}\right)\quad (-1\leq j\leq u-1)\\
    &C_j=\left(\displaystyle\bigcup_{k=0}^{j-1}\{(v_{3k+1})v_{3k+2}, (v_{3k+2})v_{3k+3}\}\right)\cup\{(v_{3j+1})v_{3j+2}\}\cup\left(\displaystyle\bigcup_{k=j+1}^{u-1}\{(v_{3k})v_{3k+1}, (v_{3k+1})v_{3k+2}\}\right)\quad (0\leq j\leq u-1)
  \end{eqnarray}

  \begin{figure}[htbp]
    \centering
    %RL.tex
    \includegraphics[scale=0.12]{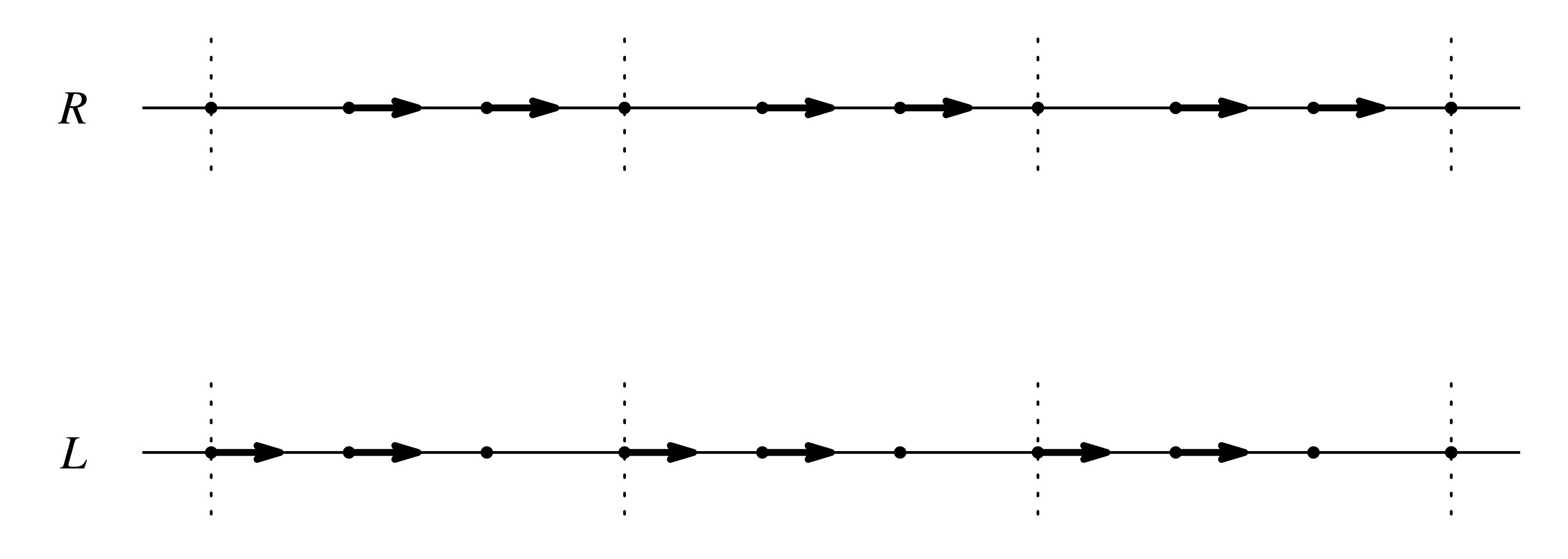}
    \label{fig:RL}
    \caption{}
  \end{figure}

  \begin{figure}[htbp]
    \centering
    %BC.tex
    \includegraphics[scale=0.125]{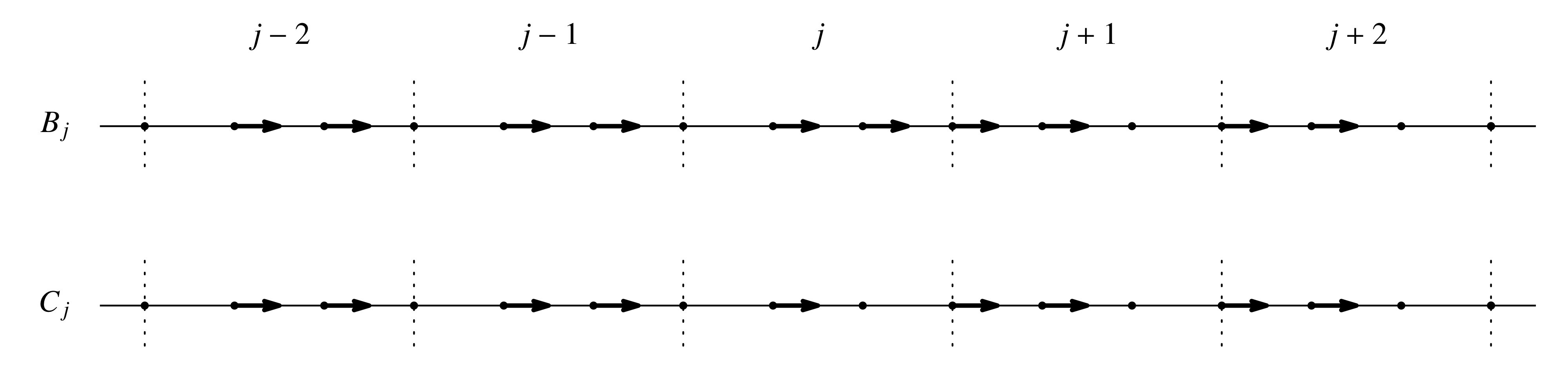}
    \caption{}
  \end{figure}

  In this case, $B_{-1}=L$ and $B_{u-1}=R$ hold.
  For any $m$, any simplex that does not contain $(v_{m-1})v_m$ and $(v_m)v_{m+1}$ is contained in $\st((v_m)v_{m-1})$, that is, in $\Delta_0$.
  Hence, every simplex in $\Delta_1$ must contain at least one of $(v_{m-1})v_m$ or $(v_m)v_{m+1}$.
  It follows that any simplex in $\Delta_1$ that is not matched in \eqref{eq:path matching} must contain exactly one of $R$, $L$, $B_j$, or $C_j$.
  We have three cases:
  \begin{enumerate}
    \item Let $t=3u$. 
    Any simplex that does not contain $(v_{3u-1})v_{3u}$ and $\textcolor{blue}{(v_{3u})c_s}$ is contained in $\st((v_{3u})v_{3u-1})$, that is, in $\Delta_0$. Therefore, any simplex that does not contain \textcolor{blue}{$(v_{3u})c_s$} must contain $(v_{3u-1})v_{3u}$.
    This implies that such a simplex must contain $R$.
    Also, any simplex containing $\textcolor{red}{(v_0)a_i}$ does not contain $(v_0)v_1$, by the definition of the Morse complex, and therefore must contain $(v_1)v_2$.
    This implies that any simplex containing $\textcolor{red}{(v_0)a_i}$ must contain exactly one of $R$, $B_j$, or $C_j$.
    We match simplices as follows: 
    \begin{equation}\label{eq:matching 3u}
      \begin{alignedat}{2}
        V\cup R &\longleftrightarrow V\cup \{\textcolor{red}{(v_0)a_1}\}\cup R & &\\ 
        V\cup\{\textcolor{blue}{(v_{3u})c_s}\}\cup B_j &\longleftrightarrow V\cup\{\textcolor{red}{(v_0)a_1}, \textcolor{blue}{(v_{3u})c_s}\}\cup B_j & \quad&(j=0, 1,\dots, u-1)\\
        V\cup\{\textcolor{blue}{(v_{3u})c_s}\}\cup C_j &\longleftrightarrow V\cup\{\textcolor{red}{(v_0)a_1}, \textcolor{blue}{(v_{3u})c_s}\}\cup C_j & \,&(j=0, 1,\dots, u-1)\\
        V\cup\{\textcolor{red}{(v_0)a_i}, \textcolor{blue}{(v_{3u})c_s}\}\cup B_j &\longleftrightarrow V\cup\{\textcolor{red}{(v_0)a_i}, \textcolor{blue}{(v_{3u})c_s}\}\cup C_j & \,&(i=2, 3,\dots, n, j=0, 1,\dots, u-1)\quad(\mathrm{see\;Figure}\;\ref{fig:4a})
      \end{alignedat}
    \end{equation}

    Combining the simplices matched in \eqref{eq:path matching} and \eqref{eq:matching 3u}, we see that $V\cup \{\textcolor{red}{(v_0)a_i}\}\cup R\,(i=2, 3,\dots, n)$ and $V\cup\{\textcolor{blue}{(v_{3u})c_s}\}\cup L\,(s=1, 2,\dots, l)$ remain unmatched.
    Since the number of elements in $V$ is $n+l$, and the number of elements in $R$ and $L$ is $2u$, the dimension of these simplices is $n+l+2u$.
    Thus, $\Delta_1$ contains $n+l-1$ unmatched simplices of dimension $n+l+2u$.
    By Theorem \ref{thm:Morse theorem}, $\M(P_{3u}\vee S_{0, n}\vee S_{0, l})\simeq (S^{n+l+2u})^{\vee(n+l-1)}$.

    \begin{figure}[htbp]
      \centering
      %0mod3.tex
      \includegraphics[scale=0.09]{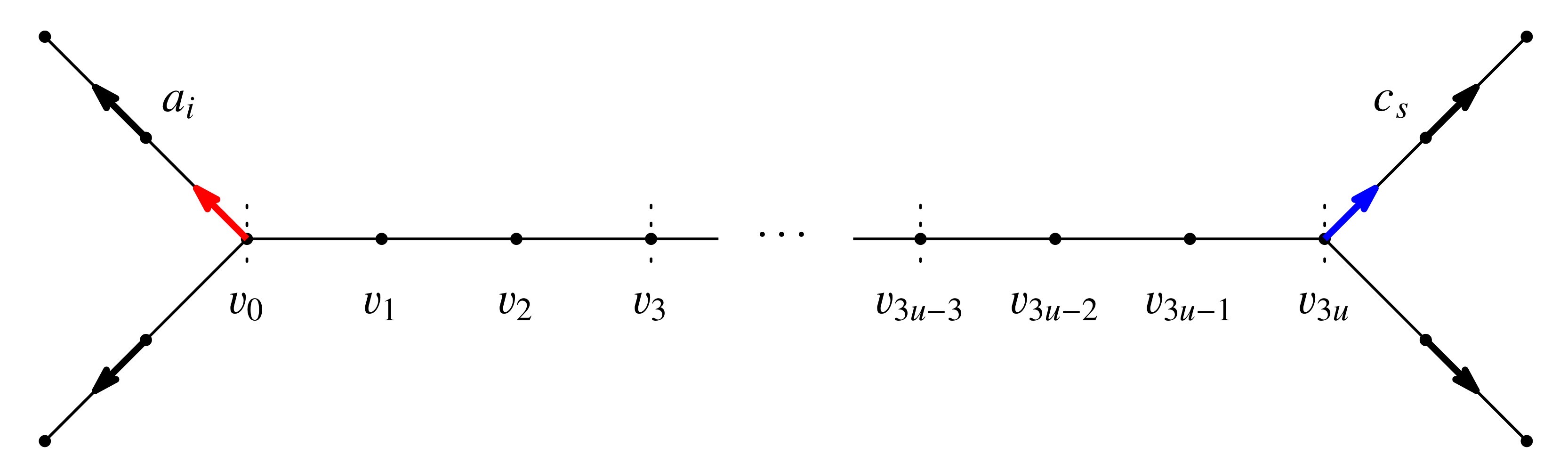}
      \caption{$P_{3u}\vee S_{0, n}\vee S_{0, l}$. The simplex generated by black $\rightarrow$ is $V$. Any simplex not containing \textcolor{blue}{$(v_{3u})c_s$} must contain $(v_{3u-1})v_{3u}$.}
      \label{fig:0mod3}
    \end{figure}

    \item Let $t=3u+1$.
    For the same reason as in (i), any simplex not containing $\textcolor{blue}{(v_{3u})c_s}$ must contain $R$. 
    Also, any simplex that does not contain $\textcolor{green}{(v_{-1})v_0}$ and $(v_0)v_1$ is contained in $\st((v_0)v_{-1})$, that is, in $\Delta_0$. Therefore, any simplex that does not contain $\textcolor{green}{(v_{-1})v_0}$ must contain $(v_0)v_1$.
    This implies that such a simplex must contain exactly one of $L$, $B_j$, or $C_j$. 
    Taking this into account, we see that any simplex that does not contain $\textcolor{green}{(v_{-1})v_0}$ and $\textcolor{blue}{(v_{3u})c_s}$ does not appear in $\Delta_1$, expect for those already matched in \eqref{eq:path matching}.
    We match simplices as follows: 
    \begin{equation}\label{eq:matching 3u+1}
      \begin{alignedat}{2}
        V\cup\{\textcolor{green}{(v_{-1})v_0}\}\cup R &\longleftrightarrow V\cup\{\textcolor{green}{(v_{-1})v_0}, \textcolor{blue}{(v_{3u})c_1}\}\cup R & &\\
        V\cup\{\textcolor{green}{(v_{-1})v_0} ,\textcolor{blue}{(v_{3u})c_1}\}\cup B_j &\longleftrightarrow V\cup\{\textcolor{green}{(v_{-1})v_0} ,\textcolor{blue}{(v_{3u})c_1}\}\cup C_{j+1} & \quad&(j=-1, 0,\dots, u-2)\quad (\mathrm{see\;Figure}\;\ref{fig:4b})\\
        V\cup\{\textcolor{green}{(v_{-1})v_0} ,\textcolor{blue}{(v_{3u})c_s}\}\cup B_j &\longleftrightarrow V\cup\{\textcolor{green}{(v_{-1})v_0},  \textcolor{blue}{(v_{3u})c_s}\}\cup C_j & \,&(s=2, 3,\dots, l,\,j=0, 1,\dots, u-1)\\
        V\cup\{\textcolor{blue}{(v_{3u})c_s}\}\cup L &\longleftrightarrow V\cup\{\textcolor{green}{(v_{-1})v_0} ,\textcolor{blue}{(v_{3u})c_s}\}\cup L & \,&(s=2, 3,\dots, l)\\
        V\cup\{\textcolor{blue}{(v_{3u})c_1}\}\cup L &\longleftrightarrow V\cup\{\textcolor{red}{(v_{-1})a_1}, \textcolor{blue}{(v_{3u})c_1}\}\cup L & &
      \end{alignedat}
    \end{equation}

    Combining the simplices matched in \eqref{eq:path matching} and \eqref{eq:matching 3u+1}, we see that $V\cup\{\textcolor{red}{(v_{-1})a_i}, \textcolor{blue}{(v_{3u})c_s}\}\cup L\,((i, s)\ne(1, 1))$ remains unmatched.
    Since the number of elements in $V$ is $n+l$, and the number of elements in $L$ is $2u$, the dimension of this simplex is $n+l+2u+1$.
    Thus, $\Delta_1$ contains $nl-1$ unmatched simplices of dimension $n+l+2u+1$.
    By Theorem \ref{thm:Morse theorem}, $\M(P_{3u+1}\vee S_{0, n}\vee S_{0, l})\simeq(S^{n+l+2u+1})^{\vee(nl-1)}$.

    \begin{figure}[htbp]
      \centering
      %1mod3.tex
      \includegraphics[scale=0.09]{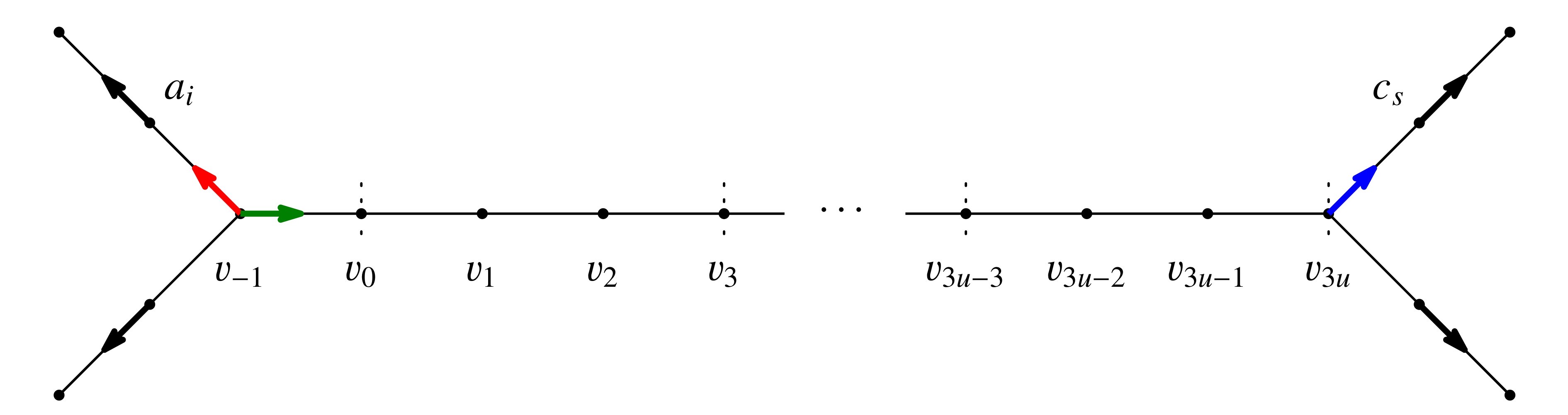}
      \caption{$P_{3u+1}\vee S_{0, n}\vee S_{0, l}$. The simplex generated by black $\rightarrow$ is $V$. Any simplex not containing \textcolor{blue}{$(v_{3u})c_s$} must contain $(v_{3u-1})v_{3u}$ and any simplex not containing \textcolor{green}{$(v_{-1})v_0$} must contain $(v_0)v_1$.}
      \label{1mod3}
    \end{figure}

    \item Let $t=3u+2$.
    For the same reasons as in (i) and (ii), any simplex not containing $\textcolor{blue}{(v_{3u})c_s}$ must contain $R$ and any simplex not containing $\textcolor{green}{(v_{-1})v_0}$ must contain exactly one of $L$, $B_j$, or $C_j$.
    Taking this into account, we see that any simplex that does not contain $\textcolor{green}{(v_{-1})v_0}$ and $\textcolor{blue}{(v_{3u})c_s}$ does not appear in $\Delta_1$, expect for those already matched in \eqref{eq:path matching}.
    Also, any simplex that does not contain $\textcolor{purple}{(v_{-2})v_{-1}}$ and $\textcolor{green}{(v_{-1})v_0}$ is contained in $\st((v_{-1})v_{-2})$, that is, in $\Delta_0$.
    Therefore, any simplex that does not contain $\textcolor{purple}{(v_{-2})v_{-1}}$ must contain $\textcolor{green}{(v_{-1})v_0}$.
    We match simplices as follows:

    \begin{equation}\label{eq:matching 3u+2}
      \begin{alignedat}{2}
        V\cup\{\textcolor{green}{(v_{-1})v_0}\}\cup R &\longleftrightarrow V\cup\{\textcolor{purple}{(v_{-2})v_{-1}}, \textcolor{green}{(v_{-1})v_0}\}\cup R & &\\
        V\cup\{\textcolor{purple}{(v_{-2})v_{-1}}, \textcolor{green}{(v_{-1})v_0}, \textcolor{blue}{(v_{3u})c_s}\}\cup B_j &\longleftrightarrow V\cup\{\textcolor{purple}{(v_{-2})v_{-1}}, \textcolor{green}{(v_{-1})v_0}, \textcolor{blue}{(v_{3u})c_s}\}\cup C_j & \,&(j=0, 1,\dots, u-1)\\
        V\cup \{\textcolor{green}{(v_{-1})v_0}, \textcolor{blue}{(v_{3u})c_s}\}\cup B_j &\longleftrightarrow V\cup \{\textcolor{green}{(v_{-1})v_0}, \textcolor{blue}{(v_{3u})c_s}\}\cup C_{j+1} & \,&(j=-1, 0,\dots, u-2)\\
        V\cup\{\textcolor{red}{(v_{-2})a_i}, \textcolor{green}{(v_{-1})v_0}, \textcolor{blue}{(v_{3u})c_s}\}\cup B_j &\longleftrightarrow V\cup\{\textcolor{red}{(v_{-2})a_i}, \textcolor{green}{(v_{-1})v_0}, \textcolor{blue}{(v_{3u})c_s}\}\cup C_{j+1} & \quad&(j=-1, 0,\dots, u-2)\\
        V\cup\{\textcolor{red}{(v_{-2})a_i}, \textcolor{green}{(v_{-1})v_0}\}\cup R &\longleftrightarrow V\cup\{\textcolor{red}{(v_{-2})a_i}, \textcolor{green}{(v_{-1})v_0}, \textcolor{blue}{(v_{3u})c_1}\}\cup R & \,&(i=1, 2,\dots, n-1)\\
        V\cup\{\textcolor{green}{(v_{-1})v_0}, \textcolor{blue}{(v_{3u})c_s}\}\cup R &\longleftrightarrow  V\cup\{\textcolor{red}{(v_{-2})a_1}, \textcolor{green}{(v_{-1})v_0}, \textcolor{blue}{(v_{3u})c_s}\}\cup R & \,&(s=2, 3,\dots, l)\\
        V\cup\{\textcolor{red}{(v_{-2})a_n}, \textcolor{green}{(v_{-1})v_0}\}\cup R &\longleftrightarrow V\cup\{\textcolor{red}{(v_{-2})a_n}, \textcolor{green}{(v_{-1})v_0}, \textcolor{blue}{(v_{3u})c_l}\}\cup R & &\\
        V\cup\{\textcolor{green}{(v_{-1})v_0}, \textcolor{blue}{(v_{3u})c_1}\}\cup R &\longleftrightarrow V\cup\{\textcolor{red}{(v_{-2})a_n}, \textcolor{green}{(v_{-1})v_0}, \textcolor{blue}{(v_{3u})c_1}\}\cup R & &\\
        V\cup\{\textcolor{purple}{(v_{-2})v_{-1}}, \textcolor{blue}{(v_{3u})c_s}\}\cup L &\longleftrightarrow V\cup\{\textcolor{purple}{(v_{-2})v_{-1}}, \textcolor{green}{(v_{-1})v_0}, \textcolor{blue}{(v_{3u})c_s}\}\cup L & &\\
      \end{alignedat}
    \end{equation}

    Combining the simplices matched in \eqref{eq:path matching} and \eqref{eq:matching 3u+2}, we see that $V\cup\{\textcolor{red}{(v_{-2})a_i}, \textcolor{green}{(v_{-1})v_0}, \textcolor{blue}{(v_{3u})c_s}\}\cup R$\,($i\ne1$, $s\ne1$ and $(i, s)\ne(n, l)$) remains unmatched.
    Since the number of elements in $V$ is $n+l$, and the number of elements in $R$ is $2u$, the dimension of this simplex is $n+l+2u+2$.
    Thus, $\Delta_1$ contains $nl-n-l=(n-1)(l-1)-1$ unmatched simplices of dimension $n+l+2u+2$.
    By Theorem \ref{thm:Morse theorem}, $\M(P_{3u+2}\vee S_{0, n}\vee S_{0, l})\simeq (S^{n+l+2u+2})^{\vee\{(n-1)(l-1)-1\}}$.

    \begin{figure}[htbp]
      \centering
      %2mod3.tex
      \includegraphics[scale=0.09]{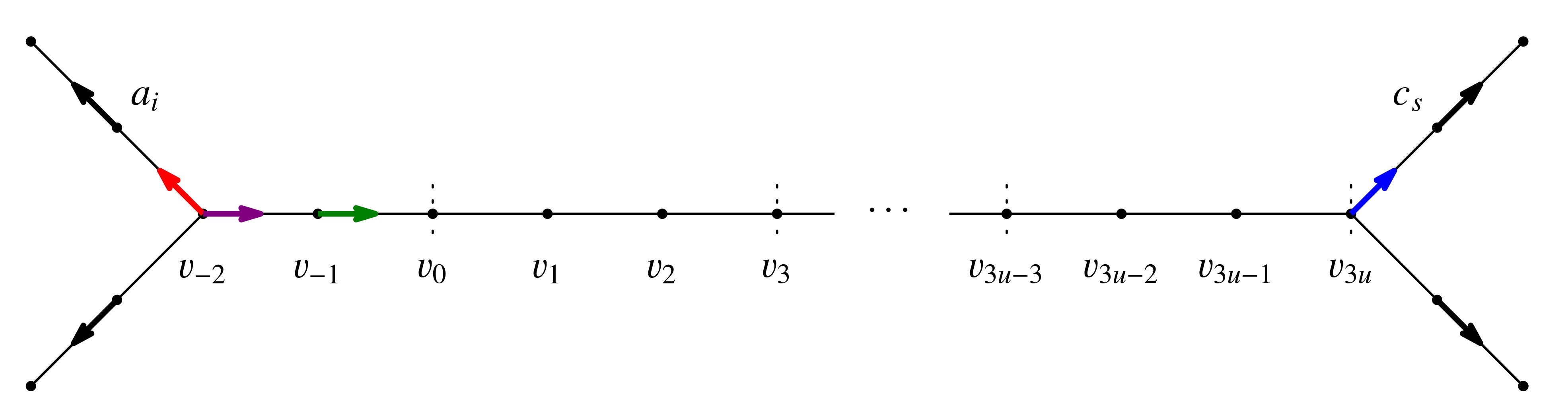}
      \caption{$P_{3u+2}\vee S_{0, n}\vee S_{0, l}$. The simplex generated by black $\rightarrow$ is $V$. Any simplex not containing \textcolor{blue}{$(v_{3u})c_s$} must contain $(v_{3u-1})v_{3u}$, any simplex not containing \textcolor{purple}{$(v_{-2})v_{-1}$} must contain \textcolor{green}{$(v_{-1})v_0$} and any simplex not containing \textcolor{green}{$(v_{-1})v_0$} must contain $(v_0)v_1$.}
      \label{fig:2mod3}
    \end{figure}
  \end{enumerate}

\end{proof}

\begin{figure}[htbp]
      \centering
      %4b.tex
      \includegraphics[scale=0.15]{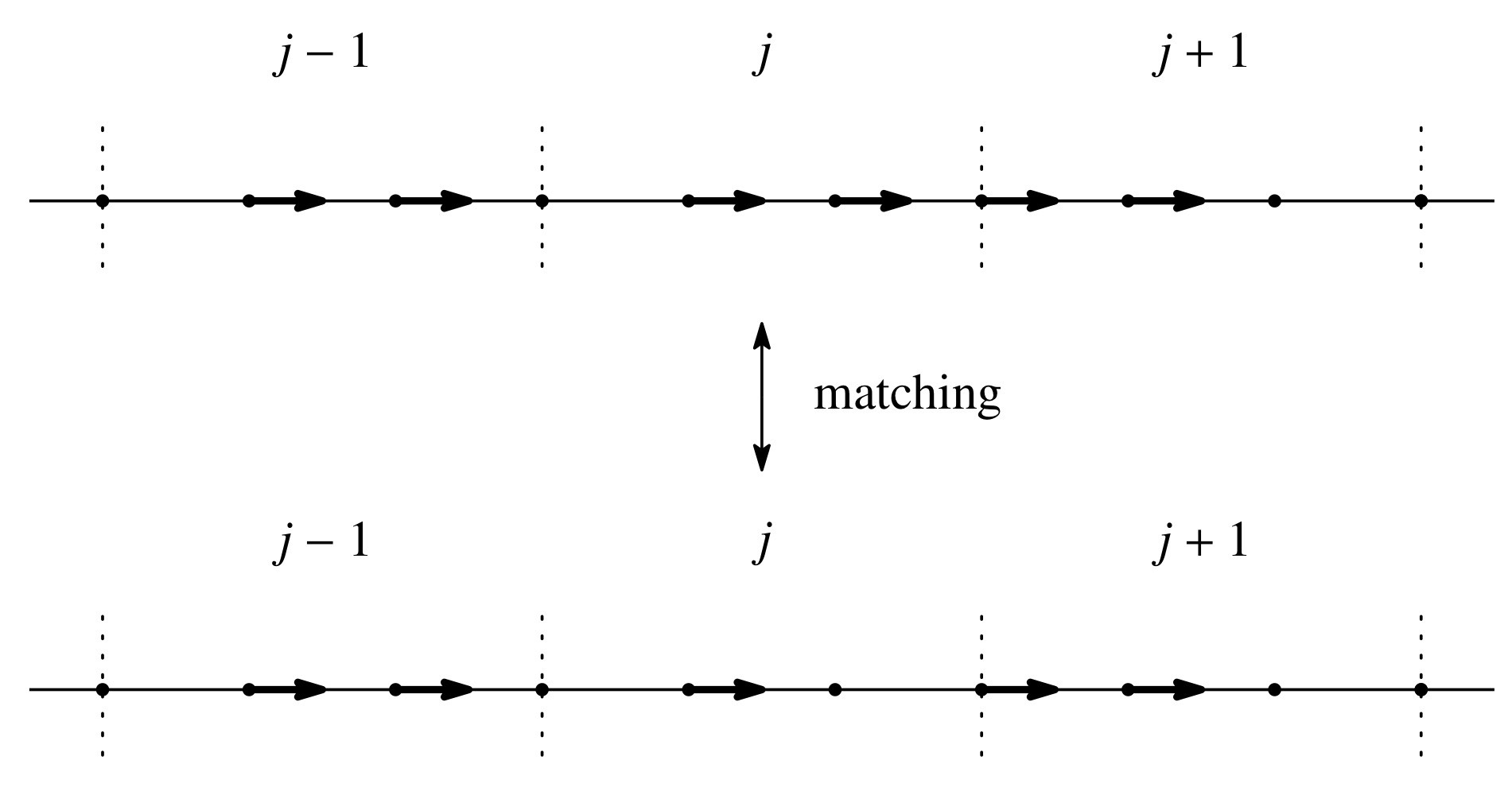}
      \caption{Match $B_j$ with $C_j\,(j=0, 1,\dots, u-1)$.}
      \label{fig:4a}
\end{figure}

\begin{figure}[htbp]
      \centering
      %4a.tex
      \includegraphics[scale=0.15]{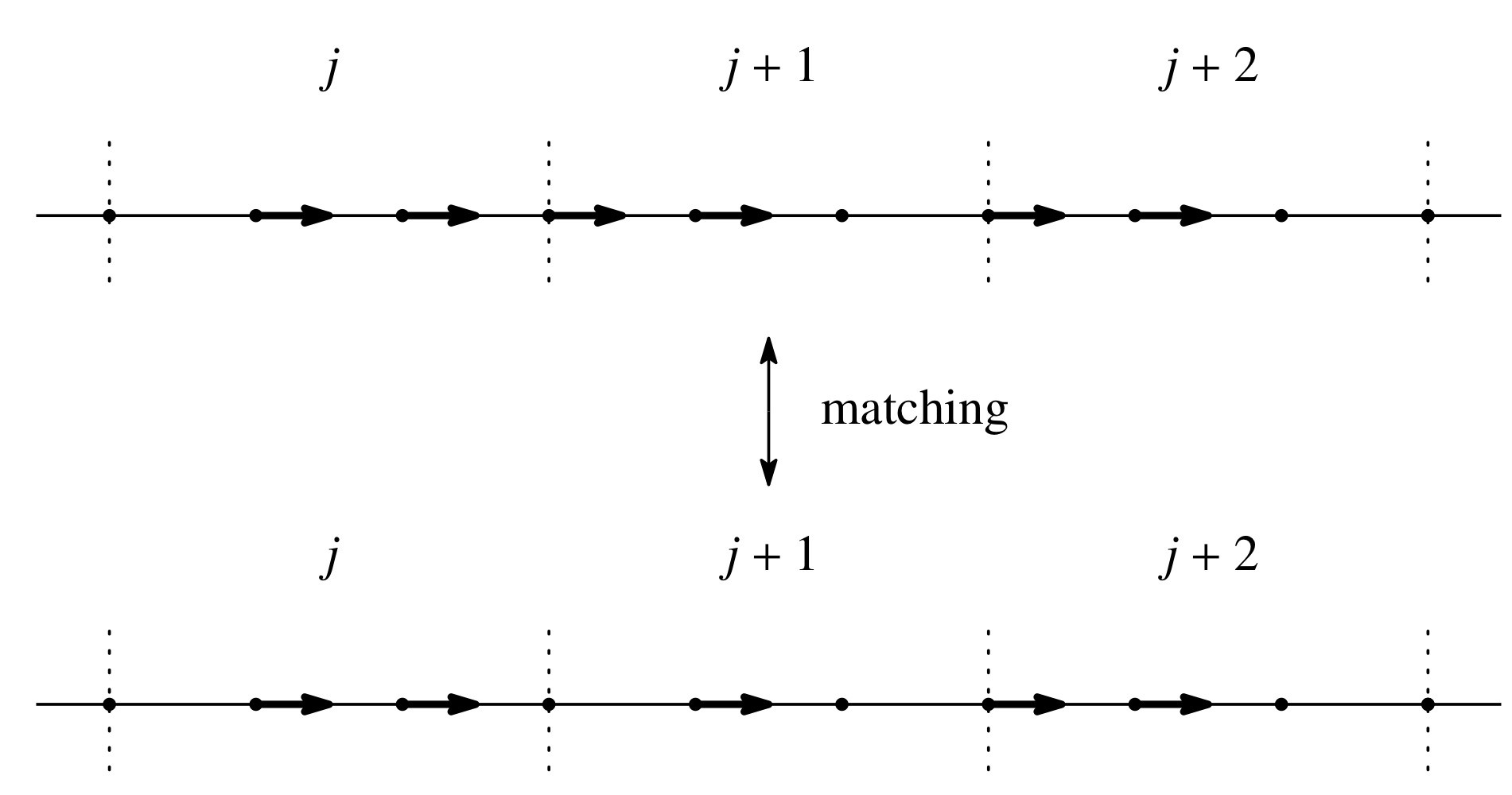}
      \caption{Match $B_j$ with $C_{j+1}\,(j=-1, 0,\dots, u-2)$.}
      \label{fig:4b}
\end{figure}

\begin{thm}\label{thm:4b}
  For any $t, n, l$, we have
  \begin{equation}
    \M(P_t\vee S_{1, n}\vee S_{1, l})\simeq
    \begin{cases}
      * &(t=3u)\\
      S^{n+l+2u+1} &(t=3u+1)\\
      S^{n+l+2u+2} &(t=3u+2)
    \end{cases}
  \end{equation}
\end{thm}

\begin{figure}[htbp]
  \centering
  %thm4b.tex
  \includegraphics[scale=0.13]{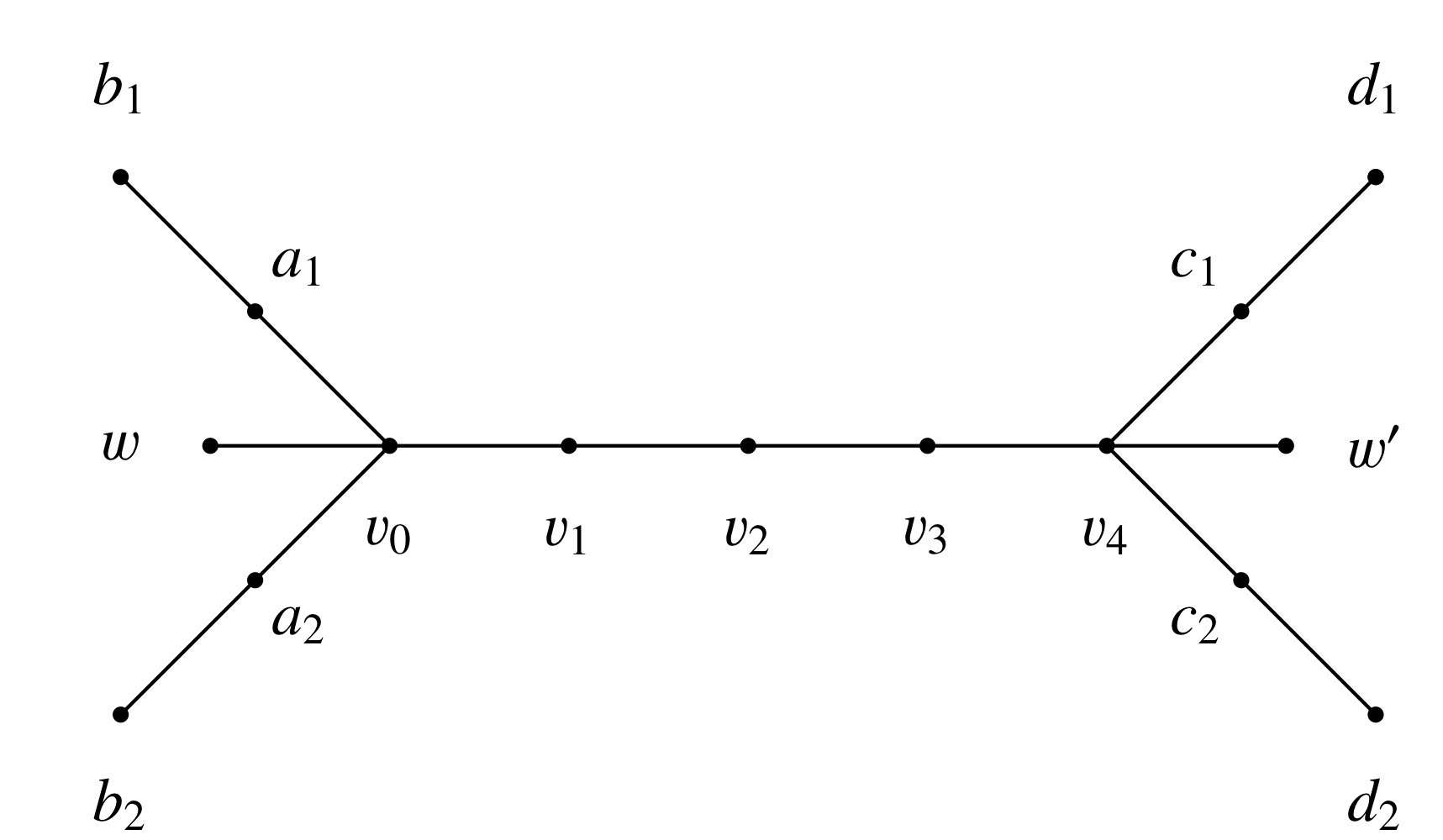}
  \caption{$P_4\vee S_{1, 2}\vee S_{1, 2}$. By Theorem \ref{thm:4b}, we see that $\M(P_4\vee S_{1, 2}\vee S_{1, 2})\simeq S^7$.}
  \label{fig:thm4b}
\end{figure}

\begin{proof}
  We label the vertices of $P_t$ as $V(P_t)=\{v_0, v_1,\dots, v_t\}$.
  Let the leaves of the extended leaves of length 2 in $S_{1, n}$ and $S_{1, l}$ be $\{a_ib_i, b_i\}$ and $\{c_sd_s, d_s\}$ respectively ($i=1, 2,\dots, n, s=1, 2,\dots, l$).
  Also, let the leaves of $S_{1, n}$ and $S_{1, l}$ be $\{v_0w, w\}$ and $\{v_tw^\prime, w^\prime\}$ respectively.
  By Lemma \ref{lem:dominated vertex}, $(w)v_0$ dominates $(v_0)a_i$ and $(v_0)v_1$, and $(w^\prime)v_t$ dominates $(v_t)c_s$ and $(v_t)v_{t-1}$.
  In the corresponding Hasse diagram $\H(P_t\vee S_{1, n}\vee S_{1, l})$, this corresponds to the removal of the edges between $v_0$ and $v_0a_i$, between $v_0$ and $v_0v_1$, between $v_t$ and $v_tc_s$ and between $v_t$ and $v_{t-1}v_t$. 
  Since these are all the edges adjacent to $v_0$ or $v_t$, the Hasse diagram consists of the following:
  \begin{itemize}
    \item the Hasse diagram of the leaf of $v_0w$, that is $\H(P_1)$ (see Figure \ref{fig:Hasse diagram1})
    \item $n$ copies of the Hasse diagram of $P_2$, obtained by connecting $v_0$, $a_i$ and $b_i$ in this order, with $v_0$ removed, that is $\H(P_2)-v_0$ (see Figure \ref{fig:Hasse diagram2})
    \item $\H(P_t)$ with both $v_0$ and $v_t$ removed (see Figure \ref{fig:Hasse diagram3})
    \item $l$ copies of the Hasse diagram of $P_2$, obtained by connecting $v_t$, $c_s$ and $d_s$ in this order, with $v_t$ removed, that is $\H(P_2)-v_t$ (see Figure \ref{fig:Hasse diagram4})
    \item the Hasse diagram of the leaf $v_tw^\prime$, that is $\H(P_1)$ (see Figure \ref{fig:Hasse diagram5})
  \end{itemize}

  \begin{figure}[htbp]
\begin{minipage}{.30\textwidth}
  \centering
  %Hasse_diagram1.tex
  \includegraphics[scale=0.07]{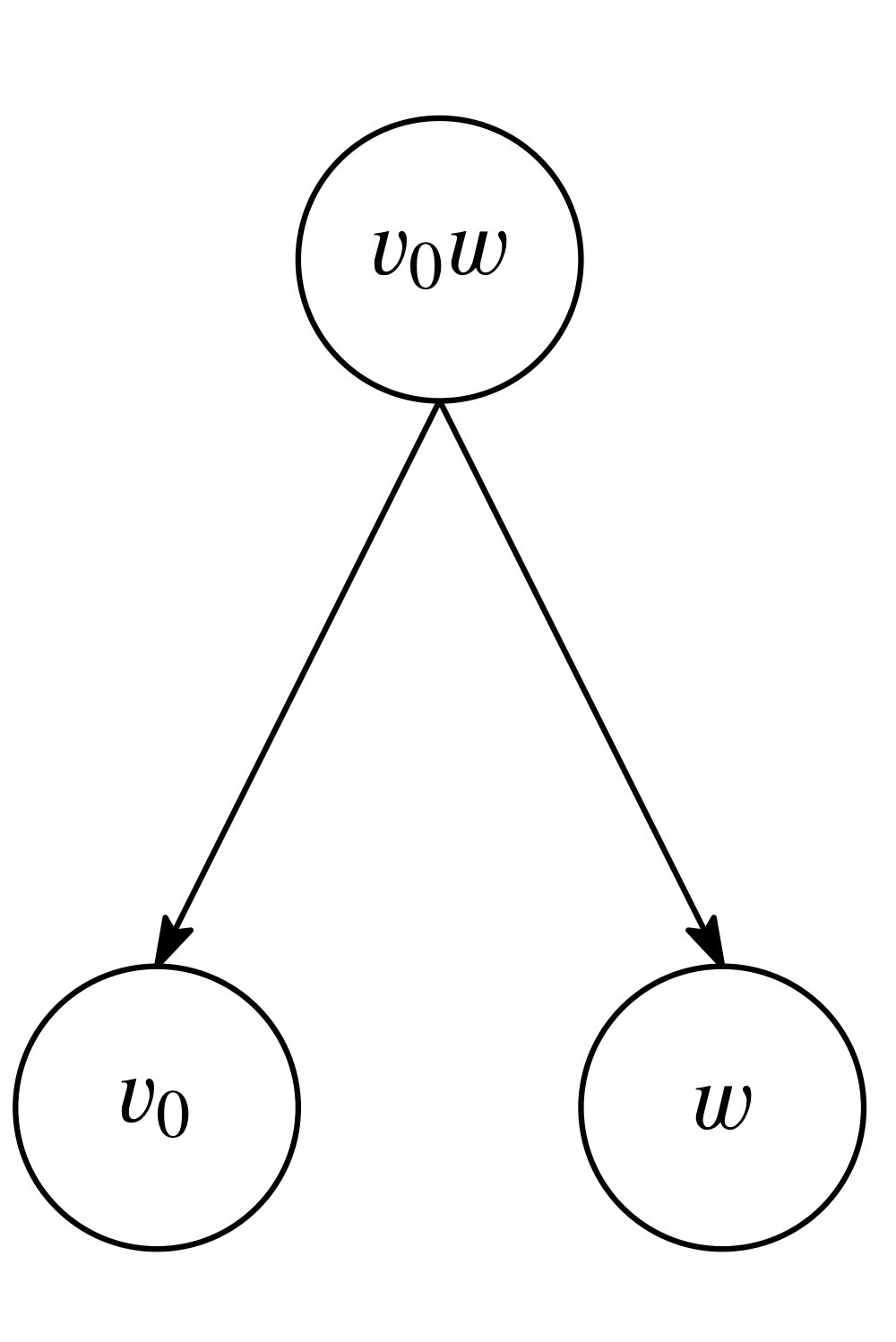}
  \caption{}
  \label{fig:Hasse diagram1}
\end{minipage}
\begin{minipage}{.39\textwidth}
  \centering
  %Hasse_diagram2.tex
  \includegraphics[scale=0.10]{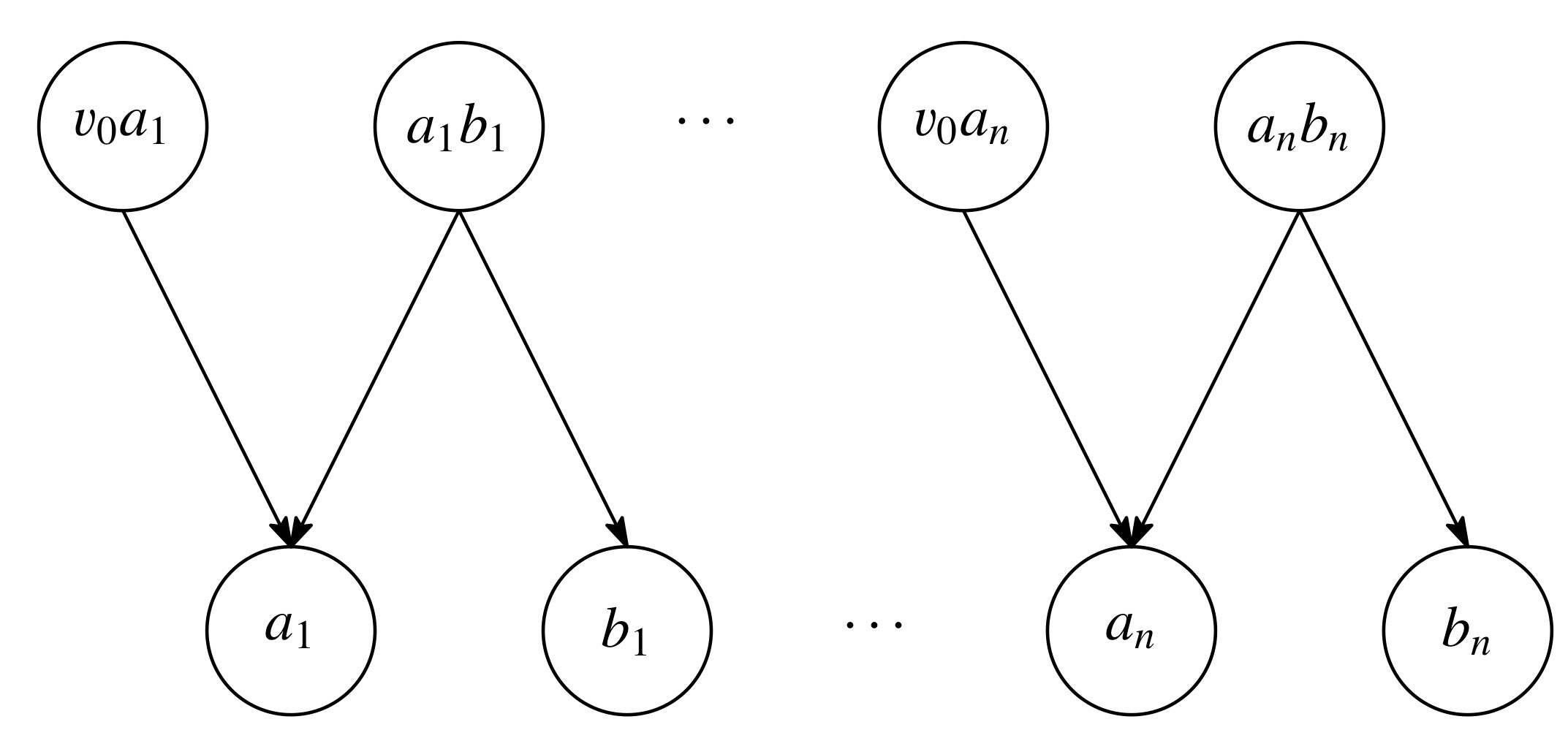}
  \caption{}
  \label{fig:Hasse diagram2}
\end{minipage}
\begin{minipage}{.30\textwidth}
  \centering
  %Hasse_diagram5.tex
  \includegraphics[scale=0.075]{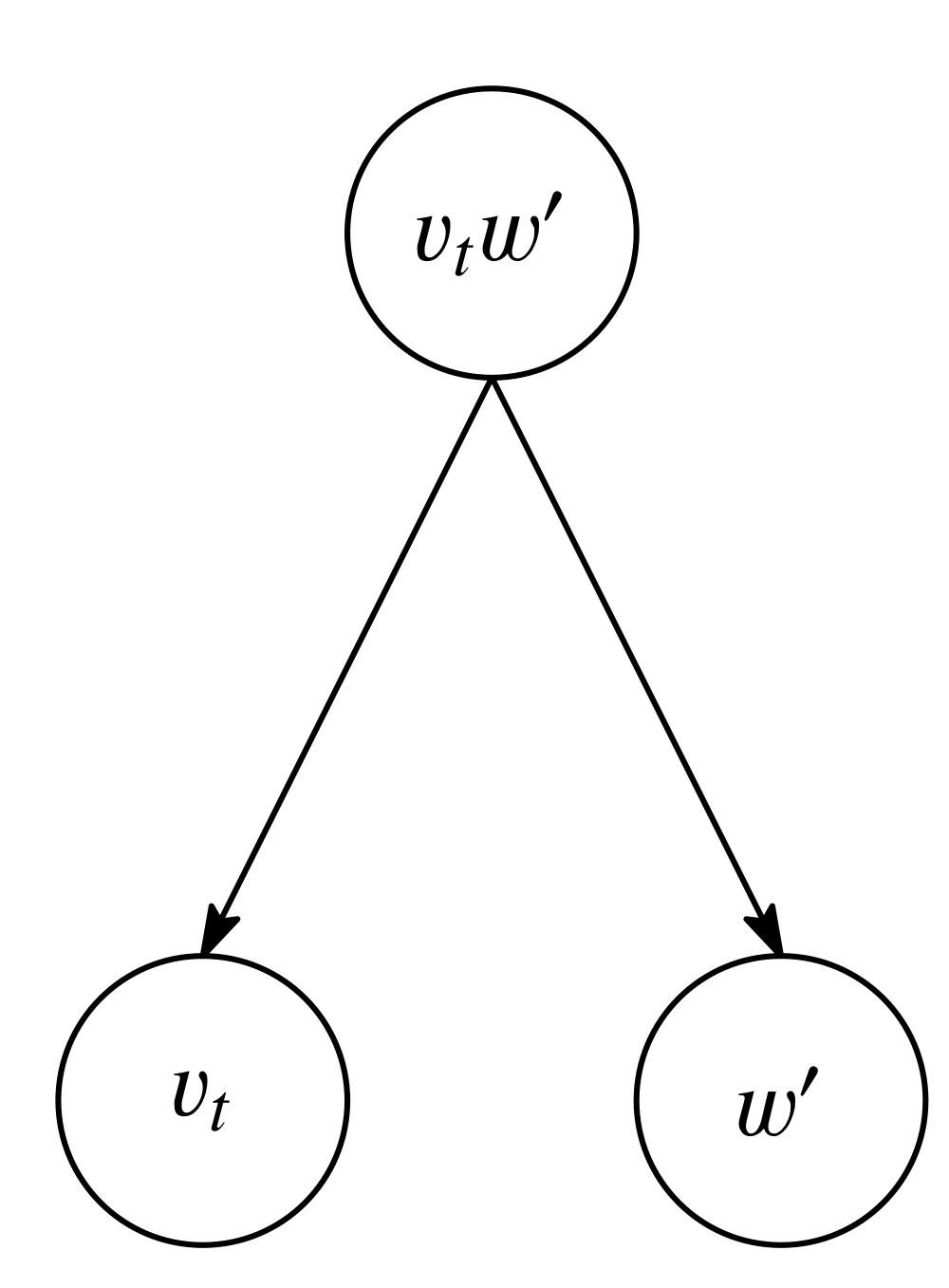}
  \caption{}
  \label{fig:Hasse diagram5}
\end{minipage}
\end{figure}

\begin{figure}[htbp]
\begin{minipage}{.45\textwidth}
  \centering
  %Hasse_diagram3.tex
  \includegraphics[scale=0.105]{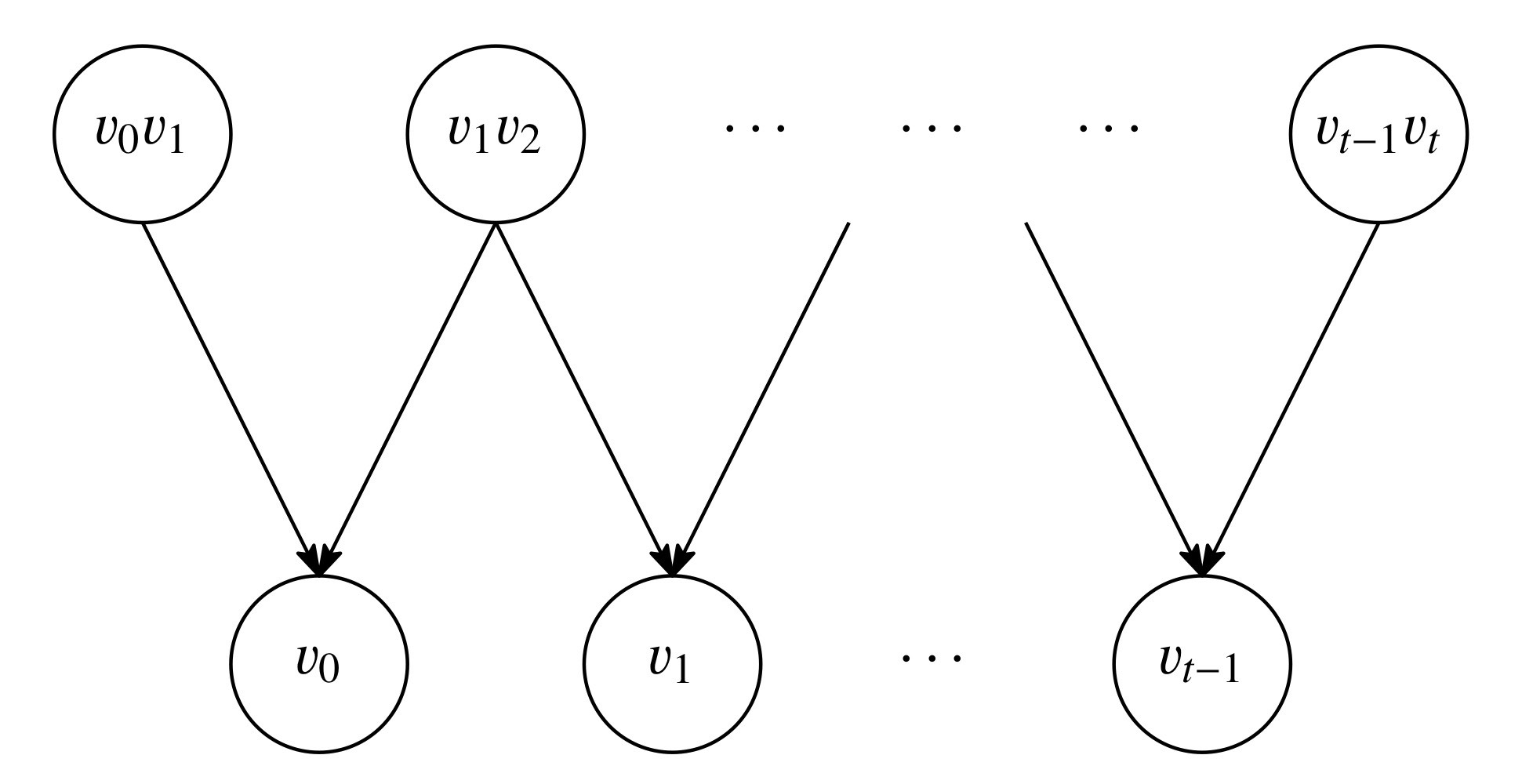}
  \caption{}
  \label{fig:Hasse diagram3}
\end{minipage}
\begin{minipage}{.45\textwidth}
  \centering
  %Hasse_diagram4.tex
  \includegraphics[scale=0.105]{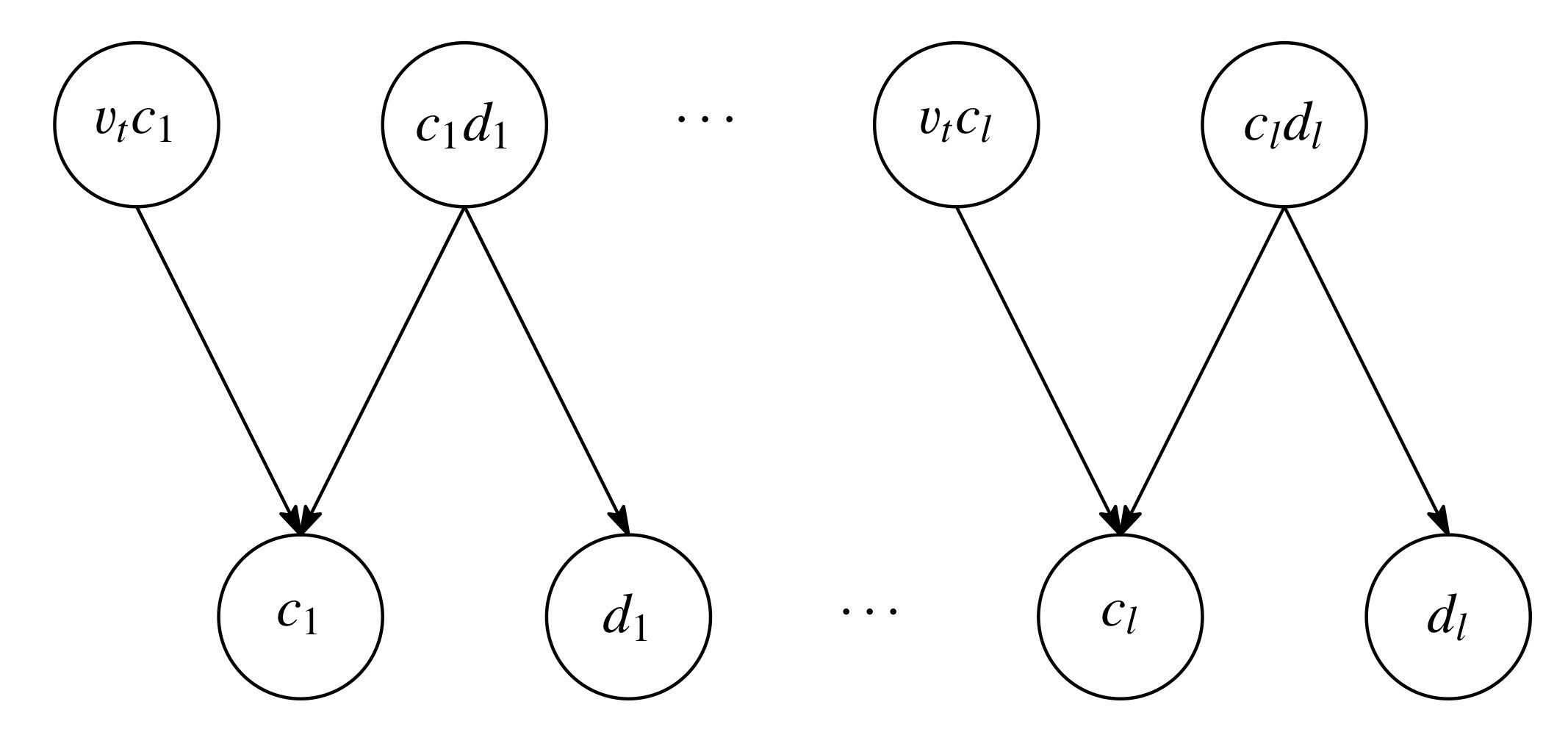}
  \caption{}
  \label{fig:Hasse diagram4}
\end{minipage}
\end{figure}

  Therefore,
  \begin{equation}
  \begin{split}
    \M(P_t\vee S_{1, n}\vee S_{1, l})&\searrow\searrow f(\H(P_1)\sqcup(\H(P_2)-v_0)^{\sqcup n}\sqcup\H(P_t-\{v_0, v_t\})\sqcup(\H(P_2)-v_t)^{\sqcup l}\sqcup\H(P_1))\\
    &\simeq f(\H(P_1))^{*2}*f(\H(P_2)-v_0)^{*n}*f(\H(P_t)-\{v_0, v_t\})*f(\H(P_2)-v_t)^{*l}
    \end{split}
  \end{equation}
  Since both $f(\H(P_2)-v_0)$ and $f(\H(P_2)-v_t)$ are homotopy equivalent to $S^0$ and $\H(P_t)-\{v_0, v_t\}\cong\H(P_{t-1})$,
  \begin{eqnarray}
    &&f(\H(P_1))^{*2}*f(\H(P_2)-v_0)^{*n}*f(\H(P_t)-\{v_0, v_t\})*f(\H(P_2)-v_t)^{*l}\\
    &\simeq& (S^0)^{*(n+l+2)}*f(\H(P_{t-1}))\simeq\Sigma^{n+l+2}\M(P_{t-1})\label{eq:4a}\\
    &\simeq&
    \begin{cases}
      * &(t=3u)\\
      S^{n+l+2u+1} &(t=3u+1)\\
      S^{n+l+2u+2} &(t=3u+2).
    \end{cases}
    \label{eq:4b}
  \end{eqnarray}
  The transformation from equation \eqref{eq:4a} to equation \eqref{eq:4b} is due to Kozlov \cite{Koz99}, as stated in Proposition \ref{prop:Kozlov} below. 
\end{proof}

\begin{prop}[{\cite[Proposition 4.6]{Koz99}}]\label{prop:Kozlov}
  Let $P_{t-1}$ be a path on $t$ vertices.
  Then,
  \begin{equation}
    \M(P_{t-1})\simeq
    \begin{cases}
      * &(t=3u)\\
      S^{2u-1} &(t=3u+1)\\
      S^{2u} &(t=3u+2)
    \end{cases}
  \end{equation}
\end{prop}

\begin{thm}\label{thm:4c}
For any $t, n, l$, we have
\begin{equation}
  \M(P_t\vee S_{1, n}\vee S_{0, l})\simeq
  \begin{cases}
    S^{n+l+2u} &(t=3u)\\
    (S^{n+l+2u+1})^{\vee l} &(t=3u+1)\\
    (S^{n+l+2u+2})^{\vee(l-1)} &(t=3u+2)
  \end{cases}
\end{equation}
\end{thm}

\begin{figure}[htbp]
  \centering
  %thm4c.tex
  \includegraphics[scale=0.13]{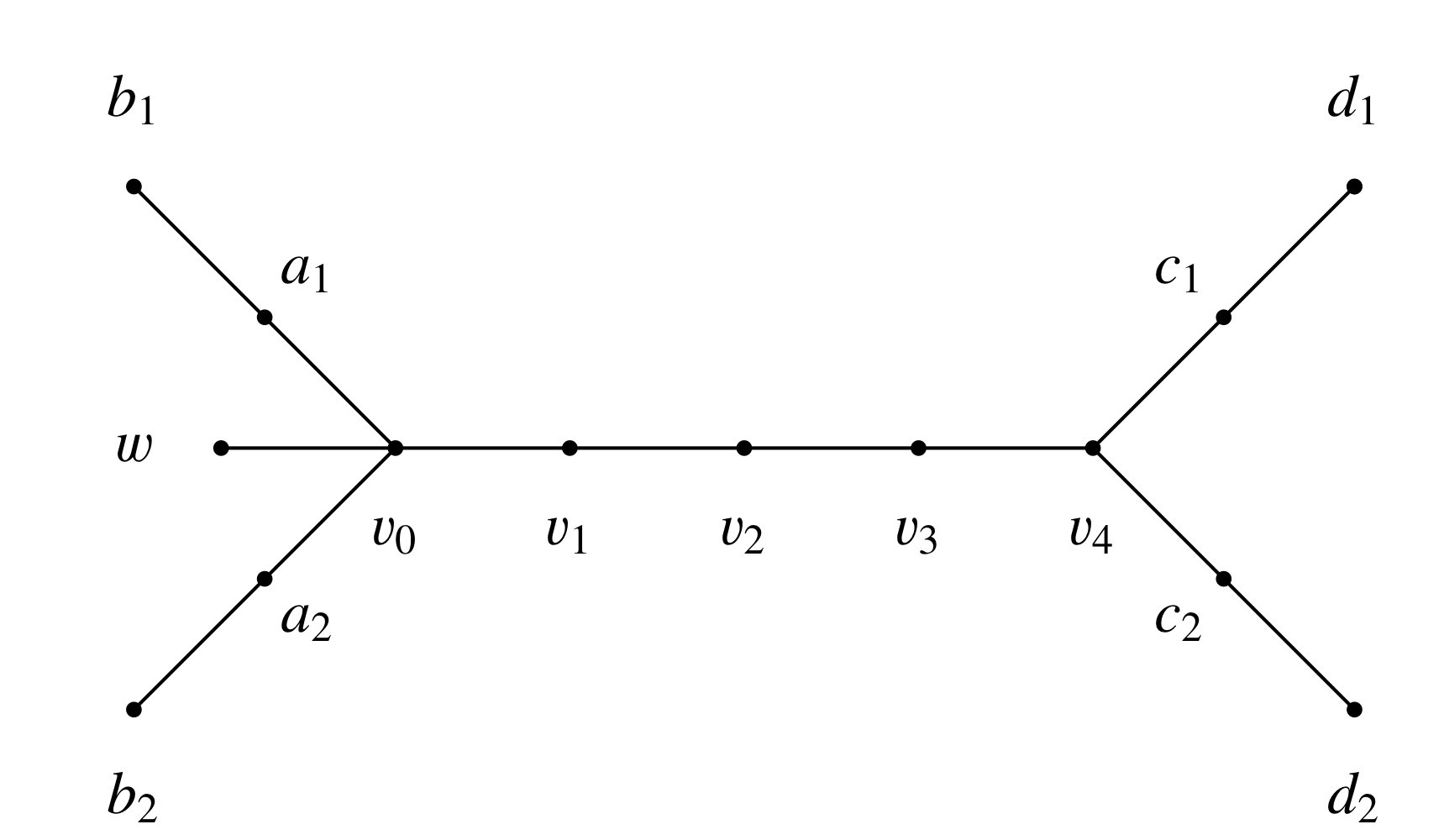}
  \caption{$P_4\vee S_{1, 2}\vee S_{0, 2}$. By Theorem \ref{thm:4c}, we see that $\M(P_4\vee S_{1, 2}\vee S_{0, 2})\simeq S^7\vee S^7$.}
  \label{fig:thm4c}
\end{figure}

\begin{proof}
  We define the vertices of $P_t\vee S_{1, n}\vee S_{0, l}$ as in the proof of Theorem \ref{thm:4b}.
  By Lemma \ref{lem:dominated vertex}, $(w)v_0$ dominates $(v_0)a_i$ and $(v_0)v_1$.
  After removing $(v_0)v_1$, every facet of $(v_2)v_1$ contains $(v_1)v_0$, and so $(v_1)v_0$ dominates $(v_2)v_1$.
  Hence, we can remove $(v_2)v_1$.
  After removing $(v_2)v_1$, every facet of $(v_3)v_4$ contains $(v_2)v_3$, and so $(v_2)v_3$ dominates $(v_3)v_4$.
  Hence, we can remove $(v_3)v_4$.
  Continuing in this manner, we see that $(v_{3k+1})v_{3k}$ dominates $(v_{3k+2})v_{3k+1}$ for all $0\leq k\leq u$, and $(v_{3k-1})v_{3k}$ dominates $(v_{3k})v_{3k+1}$ for all $0<k\leq u$.
  In the corresponding Hasse diagram $\H(P_t\vee S_{1, n}\vee S_{0, l})$, this corresponds to the removal of the edges between $v_0$ and $v_0a_i$, between $v_0$ and $v_0v_1$, between $v_{3k}$ and $v_{3k}v_{3k+1}$ and between $v_{3k+2}$ and $v_{3k+1}v_{3k+2}$.
  We have three cases:
  \begin{enumerate}
    \item Let $t=3u$. 
    In this case, the Hasse diagram consists of the following:
    \begin{itemize}
      \item the Hasse diagram of the leaf $v_0w$, that is $\H(P_1)$ (see Figure \ref{fig:Hasse diagram1})
      \item $n$ copies of the Hasse diagram of $P_2$, obtained by connecting $v_0$, $a_i$ and $b_i$ in this order, with $v_0$ removed, that is $\H(P_2)-v_0$ (see Figure \ref{fig:Hasse diagram2})
      \item $(2u-1)$ copies of the Hasse diagram $\H(P_1)$ (see Figure \ref{fig:Hasse diagram6})
      \item $\H(S_{1, l})$ (see Figure \ref{fig:Hasse diagram7})
    \end{itemize}
    Therefore,
    \begin{equation}
      \begin{split}
        \M(P_{3u}\vee S_{1, n}\vee S_{0, l})&\searrow\searrow f(\H(P_1)\sqcup(\H(P_2)-v_0)^{\sqcup n}\sqcup(\H(P_1))^{\sqcup(2u-1)}\sqcup\H(S_{1, l}))\\
        &\simeq f(\H(P_1))^{*2u}*f(\H(P_2)-v_0)^{*n}*f(\H(S_{1, l}))\\
        &\simeq f(\H(P_1))^{*2u}*f(\H(P_2)-v_0)^{*n}*\M(S_{1, l})
      \end{split}
    \end{equation}
    By Theorem \ref{thm:3b},
    \begin{equation}
    \begin{split}
      f(\H(P_1))^{*2u}*f(\H(P_2)-v_0)^{*n}*\M(S_{1, l})&\simeq (S^0)^{*2u}*(S^0)^{*n}*S^l\\
      &\simeq\Sigma^{n+2u}S^l\\
      &\simeq S^{n+l+2u}.
    \end{split}
    \end{equation}

    \begin{figure}[htbp]
    \begin{minipage}{.45\textwidth}
      \centering
      %Hasse_diagram6.tex
      \includegraphics[scale=0.10]{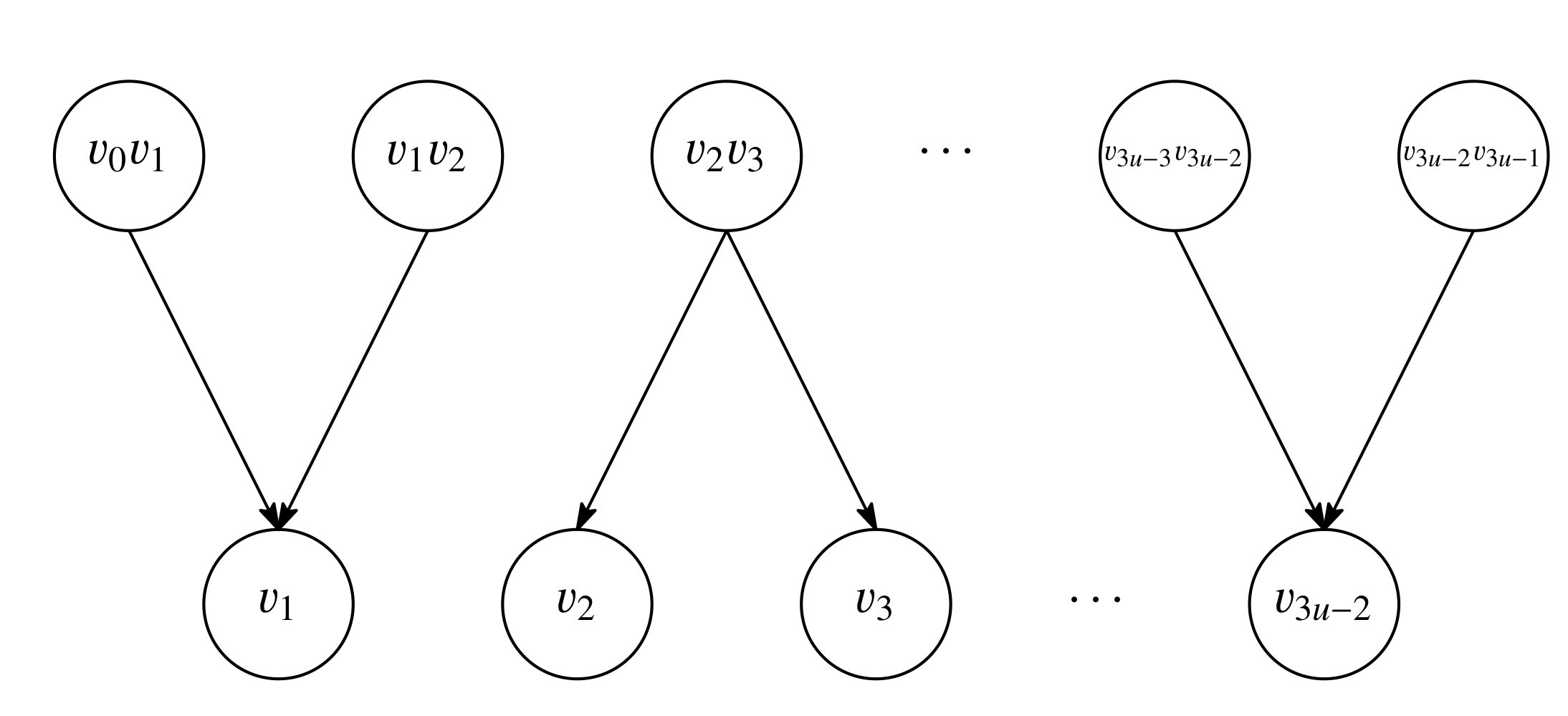}
      \caption{}
      \label{fig:Hasse diagram6}
    \end{minipage}
    \begin{minipage}{.45\textwidth}
      \centering
      %Hasse_diagram7.tex
      \includegraphics[scale=0.10]{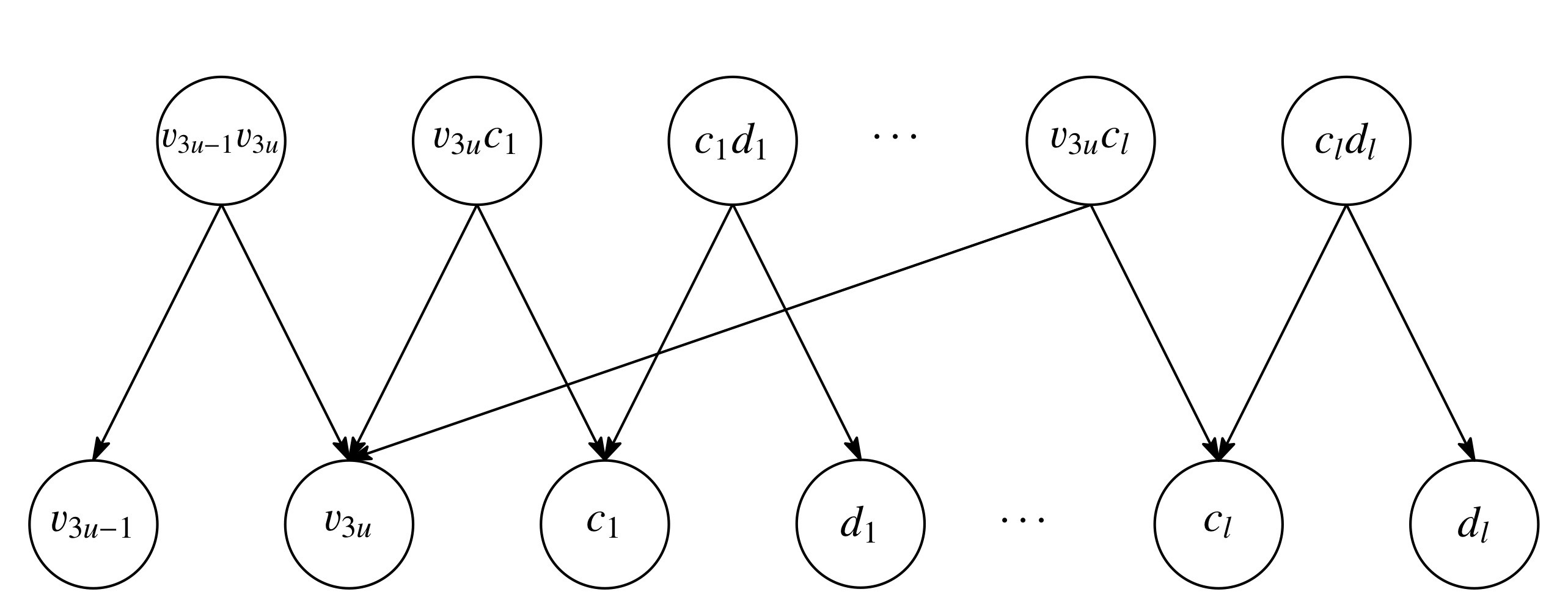}
      \caption{}
      \label{fig:Hasse diagram7}
    \end{minipage}
    \end{figure}

    \item Let $t=3u+1$.
    We now consider the Hasse diagram in the case where the edge between $v_{3u}$ and $v_{3u}v_{3u+1}$ is not removed.
    Then, the Hasse diagram consists of the following: 
    \begin{itemize}
      \item the Hasse diagram of the leaf $v_0w$, that is $\H(P_1)$ (see Figure \ref{fig:Hasse diagram1})
      \item $n$ copies of the Hasse diagram of $P_2$, obtained by connecting $v_0$, $a_i$ and $b_i$ in this order, with $v_0$ removed, that is $\H(P_2)-v_0$ (see Figure \ref{fig:Hasse diagram2})
      \item $(2u-1)$ copies of the Hasse diagram $\H(P_1)$ (see Figure \ref{fig:Hasse diagram6})
      \item $\H(S_{0, l+1})$ (see Figure \ref{fig:Hasse diagram8})
    \end{itemize}

    \begin{figure}[htbp]
      \centering
      %Hasse_diagram8.tex
      \includegraphics[scale=0.10]{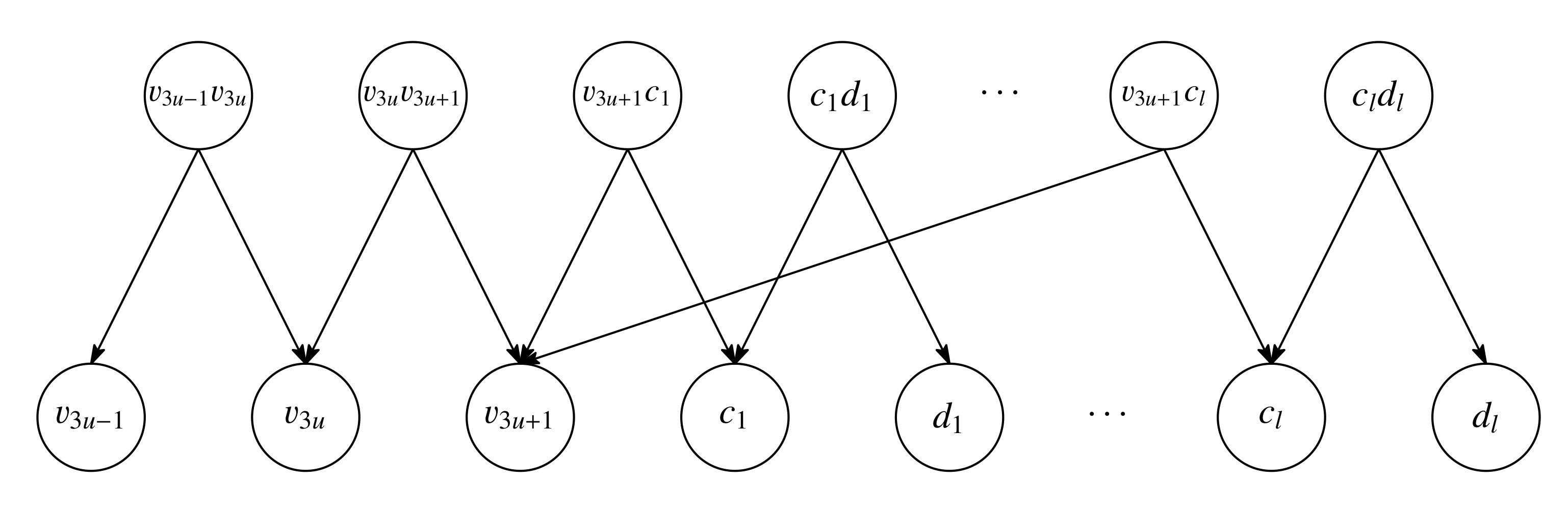}
      \caption{}
      \label{fig:Hasse diagram8}
    \end{figure}

    Therefore,
    \begin{equation}
      \begin{split}
        \M(P_{3u+1}\vee S_{1, n}\vee S_{0, l})&\searrow\searrow f(\H(P_1)\sqcup(\H(P_2)-v_0)^{\sqcup n}\sqcup(\H(P_1))^{\sqcup(2u-1)}\sqcup\H(S_{0, l+1}))\\
        &\simeq f(\H(P_1))^{*2u}*f(\H(P_2)-v_0)^{*n}*f(\H(S_{0, l+1}))\\
        &\simeq f(\H(P_1))^{*2u}*f(\H(P_2)-v_0)^{*n}*\M(S_{0, l+1})
      \end{split}
    \end{equation}
    By Theorem \ref{thm:3a},
    \begin{equation}
    \begin{split}
      f(\H(P_1))^{*2u}*f(\H(P_2)-v_0)^{*n}*\M(S_{0, l+1})&\simeq (S^0)^{*2u}*(S^0)^{*n}*(S^{l+1})^{\vee l}\\
      &\simeq\Sigma^{n+2u}(S^{l+1})^{\vee l}\\
      &\simeq (S^{n+l+2u+1})^{\vee l}.
    \end{split}
    \end{equation}

    \item Let $t=3u+2$.
    In this case, the Hasse diagram consists of the following: 
    \begin{itemize}
      \item the Hasse diagram of the leaf $v_0w$, that is $\H(P_1)$ (see Figure \ref{fig:Hasse diagram1})
      \item $n$ copies of the Hasse diagram of $P_2$, obtained by connecting $v_0$, $a_i$ and $b_i$ in this order, with $v_0$ removed, that is $\H(P_2)-v_0$ (see Figure \ref{fig:Hasse diagram2})
      \item $(2u+1)$ copies of the Hasse diagram $\H(P_1)$ (see Figure\ref{fig:Hasse diagram9})
      \item $\H(S_{0, l})$ (see Figure \ref{fig:Hasse diagram10})
    \end{itemize}
    \begin{figure}[htbp]
    \begin{minipage}{.45\textwidth}
      \centering
      %Hasse_diagram9.tex
      \includegraphics[scale=0.095]{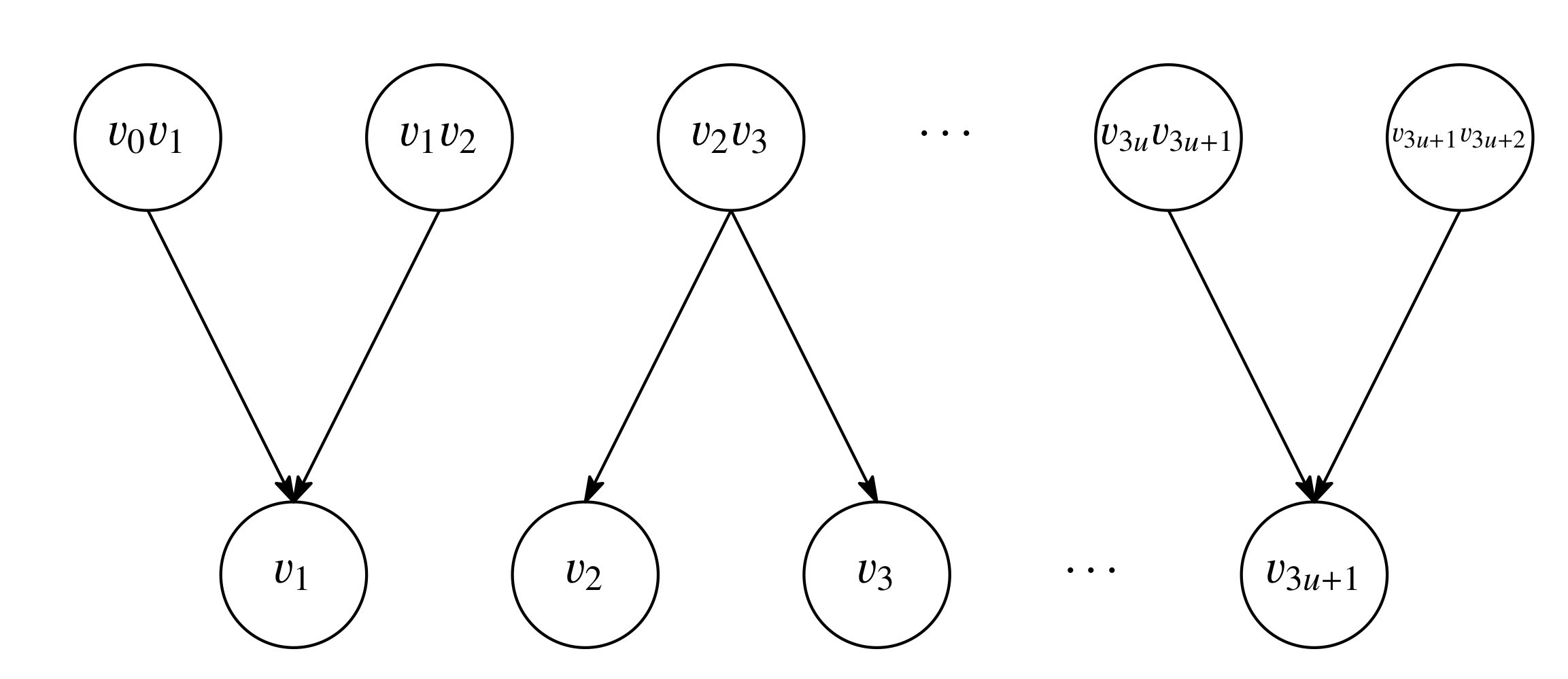}
      \caption{}
      \label{fig:Hasse diagram9}
    \end{minipage}
    \begin{minipage}{.45\textwidth}
      \centering
      %Hasse_diagram10.tex
      \includegraphics[scale=0.095]{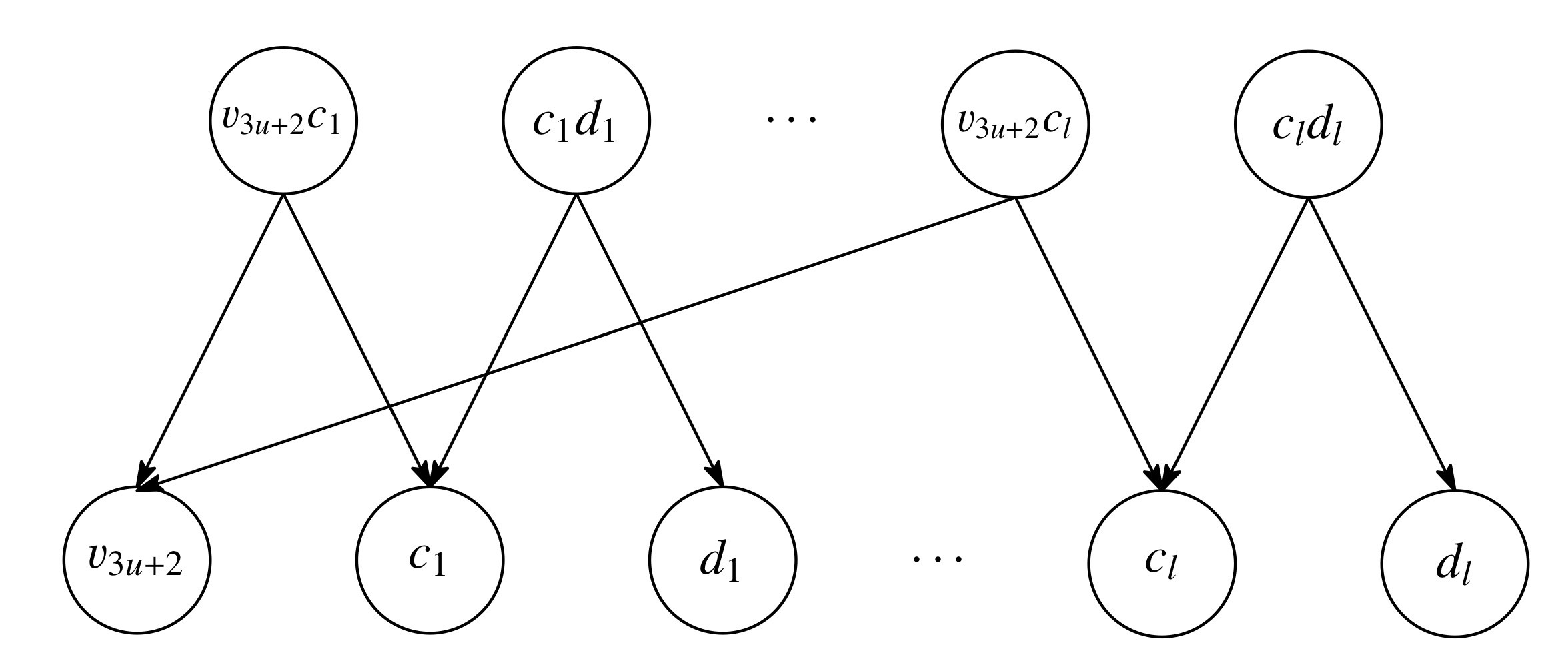}
      \caption{}
      \label{fig:Hasse diagram10}
    \end{minipage}
    \end{figure}

    Therefore,
    \begin{equation}
      \begin{split}
        \M(P_{3u+2}\vee S_{1, n}\vee S_{0, l})&\searrow\searrow f(\H(P_1)\sqcup(\H(P_2)-v_0)^{\sqcup n}\sqcup(\H(P_1))^{\sqcup(2u+1)}\sqcup\H(S_{0, l}))\\
        &\simeq f(\H(P_1))^{*(2u+2)}*f(\H(P_2)-v_0)^{*n}*f(\H(S_{0, l}))\\
        &\simeq f(\H(P_1))^{*(2u+2)}*f(\H(P_2)-v_0)^{*n}*\M(S_{0, l})
      \end{split}
    \end{equation}
    By Theorem \ref{thm:3a},
    \begin{equation}
    \begin{split}
      f(\H(P_1))^{*(2u+2)}*f(\H(P_2)-v_0)^{*n}*\M(S_{0, l})&\simeq (S^0)^{*(2u+2)}*(S^0)^{*n}*(S^l)^{\vee(l-1)}\\
      &\simeq\Sigma^{n+2u+2}(S^l)^{\vee(l-1)}\\
      &\simeq (S^{n+l+2u+2})^{\vee(l-1)}.
      \end{split}
    \end{equation}

  \end{enumerate}
\end{proof}

\end{document}